\UseRawInputEncoding
\documentclass[a4paper]{amsart}
\usepackage{graphicx}
\usepackage{amssymb}
\usepackage{amsmath}
\usepackage{amsthm}
\usepackage{amscd}
\usepackage{color}
\usepackage{colordvi}
\usepackage[all,2cell]{xy}
\usepackage{picture}
\usepackage{multirow}
\usepackage{makecell}

\usepackage{amsfonts,enumerate,verbatim,mathtools,tikz,bm,tikz-cd,hyperref,comment}
\hypersetup{
	colorlinks=true,
	linkcolor=blue,
	filecolor=magenta,
	urlcolor=cyan,
}

\UseAllTwocells \SilentMatrices

\title[An introduction to model categories with examples]{An introduction to model categories with examples}

\thanks{Date: \today}

\author[Xiao-Wu Chen]{Xiao-Wu Chen}
\address{
School of Mathematical Sciences, University of Science and Technology of China, Hefei 230026, Anhui, PR China}
\email{xwchen@mail.ustc.edu.cn}

\theoremstyle{plain}
\newtheorem{lem}{Lemma}[section]
\newtheorem{prop}[lem]{Proposition}
\newtheorem{cor}[lem]{Corollary}
\newtheorem{thm}[lem]{Theorem}

\theoremstyle{remark}

\theoremstyle{definition}
\newtheorem{rem}[lem]{Remark}
\newtheorem{exm}[lem]{Example}
\newtheorem{defn}[lem]{Definition}

\numberwithin{equation}{section}

\makeindex

\begin{document}

\begin{abstract}
We give an informal introduction to model categories, and treat three important examples in some details: the category of small categories, the category of dg algebras, and the category of small dg categories.
\end{abstract}

\maketitle

\setcounter{tocdepth}{2}
\tableofcontents

\section{Introduction}

Homotopy is an important concept in modern mathematics, especially in topology and algebra. A model category means a category with a specific model structure, and provides a general framework to develop a homotopy theory. For example, in dealing with natural isomorphisms between functors, we are really doing a homotopy theory. For another example, a model structure on the category of small dg categories allows us to talk about a meaningful homotopy between dg functors.

In this note, we give an informal introduction to various model structures appearing naturally in algebra. The most important examples are the ones on the category of small categories, on the category of dg algebras, and on the category of small dg categories, respectively. This note might serve as a detailed  introduction to the Dwyer-Kan model structure  on the category of small dg categories.

\section{Preliminaries on categories}

Throughout this section,  we fix a category $\mathcal{C}$. We denote by ${\rm Obj}(\mathcal{C})$ and ${\rm Mor}(\mathcal{C})$ the class of objects and the class of morphisms in $\mathcal{C}$, respectively. For two objects $X$ and $Y$, we denote by $\mathcal{C}(X, Y)$ the Hom set formed by all the  morphisms from $X$ to $Y$.

\subsection{Factor categories and localizations}

By an \emph{equivalence relation} $\mathcal{R}$  on $\mathcal{C}$, we mean
$$\mathcal{R}=(\mathcal{R}_{X, Y})_{X, Y\in {\rm Obj}(\mathcal{C})},$$
 which consists of equivalence relations $\mathcal{R}_{X, Y}$ on $\mathcal{C}(X, Y)$ for all objects $X$ and $Y$; the equivalence relations $\mathcal{R}_{X, Y}$ are subject to the following condition: for any $(\alpha, \beta)\in \mathcal{R}_{X, Y}$, $f\colon X'\rightarrow X$ and $g\colon Y\rightarrow Y'$, the element $(g\circ \alpha\circ f, g\circ \beta \circ f)$ belongs to $\mathcal{R}_{X', Y'}$.

The \emph{factor category} $\mathcal{C}/\mathcal{R}$ is defined such that ${\rm Obj}(\mathcal{C}/\mathcal{R})={\rm Obj}(\mathcal{C})$, and that $(\mathcal{C}/\mathcal{R})(X, Y)=\mathcal{C}(X, Y)/{\mathcal{R}_{X, Y}}$ for any objects $X$ and $Y$. Here, $\mathcal{C}(X, Y)/{\mathcal{R}_{X, Y}}$ denotes the quotient set of $\mathcal{C}(X, Y)$ with respect to the equivalence relation $\mathcal{R}_{X, Y}$. The composition of morphisms in $\mathcal{C}/\mathcal{R}$ is induced from the one in $\mathcal{C}$. The canonical functor
$${\rm can}\colon \mathcal{C}\longrightarrow \mathcal{C}/\mathcal{R}$$
 acts on objects by the identity; it is clearly full and dense.

 Let $F\colon \mathcal{C}\rightarrow \mathcal{D}$ be a functor. Denote by ${\rm Im}(F)$ the \emph{essential image} of $F$, that is, the full subcategory of $\mathcal{D}$ formed by those objects  that are isomorphic to $F(C)$ for some object $C\in \mathcal{C}$. Denote by ${\rm inc}\colon {\rm Im}(F)\rightarrow \mathcal{D}$ the inclusion functor. Associated to $F$, there is an equivalence relation $\mathcal{R}(F)$ on $\mathcal{C}$, which is determined by
 $$\mathcal{R}(F)_{X, Y}=\{(\alpha, \beta)\; |\; \alpha, \beta\in \mathcal{C}(X, Y) \mbox{ with } F(\alpha)=F(\beta)\}.$$
 There is a unique functor $\widetilde{F}\colon \mathcal{C}/{\mathcal{R}(F)}\rightarrow {\rm Im}(F)$ making the following diagram commute.
 \[\xymatrix{
 \mathcal{C}\ar[d]_-{\rm can} \ar[rr]^{F} && \mathcal{D}\\
 \mathcal{C}/{\mathcal{R}(F)} \ar@{.>}[rr]^-{\widetilde{F}} && {\rm Im}(F) \ar[u]_-{\rm inc}
 }\]
 This diagram might be called the \emph{standard factorization} of $F$.

 \begin{prop}
 The above functor $\widetilde{F}\colon \mathcal{C}/{\mathcal{R}(F)}\rightarrow {\rm Im}(F)$ is faithful and dense. Moreover, it is an equivalence if and  only if the given functor $F$ is full. \hfill $\square$
 \end{prop}

\begin{rem}\label{rem:image}
There is an alternative manner to describe the essential image of a functor, up to an equivalence. Let $F\colon \mathcal{C}\rightarrow \mathcal{D}$ be a functor. We define a new category $\mathcal{C}_F$ such that ${\rm Obj}(\mathcal{C}_F)={\rm Obj}(\mathcal{C})$ and $\mathcal{C}_F(C_1, C_2)=\mathcal{D}(F(C_1), F(C_2))$, whose composition of morphisms is induced by the one in $\mathcal{D}$. We have a natural factorization of $F$ as follows:
$$\mathcal{C}\stackrel{F_1}\longrightarrow \mathcal{C}_F \stackrel{F_2}\longrightarrow \mathcal{D},$$
where $F_1$ acts on objects by the identity, and $F_2$ acts on morphisms by the identity. We observe that $F_2$ induces an equivalence
 $$\mathcal{C}_F\stackrel{\sim}\longrightarrow {\rm Im}(F).$$
\end{rem}

Let $\mathcal{S}$ be a class of  morphisms in $\mathcal{C}$. By a \emph{localization} of $\mathcal{C}$ with respect to $\mathcal{S}$, we mean a category $\mathcal{C}[\mathcal{S}^{-1}]$ together with a functor $q\colon \mathcal{C}\rightarrow \mathcal{C}[\mathcal{S}^{-1}]$ satisfying the following conditions:
\begin{enumerate}
\item[(L1)] $F(f)$ is an isomorphism in $\mathcal{C}[\mathcal{S}^{-1}]$ for any $f\in \mathcal{S}$;
\item[(L2)] for any functor $G\colon \mathcal{C}\rightarrow \mathcal{D}$ sending all the elements in $\mathcal{S}$ to isomorphisms, there exists a unique functor $G'\colon \mathcal{C}[\mathcal{S}^{-1}]\rightarrow \mathcal{D}$ such that $G=G'q$.
\end{enumerate}
If such a localization exists, it is necessarily unique up to a unique isomorphism. For details, we refer to \cite[Chapter~1]{GZ}.

\subsection{Liftings and orthogonality} Let $\mathcal{C}$ be a category. A morphism $f\colon X\rightarrow Y$ is said to be a \emph{retract} of another morphism $f'\colon X'\rightarrow Y'$, if there exists a commutative diagram
\[
\xymatrix{
X\ar[d]_-f \ar[r]^-i & X' \ar[d]^-{f'} \ar[r]^-p & X \ar[d]^-{f}\\
Y\ar[r]^-j & Y' \ar[r]^-q & Y
}
\]
satisfying $p\circ i={\rm Id}_X$ and $q\circ j={\rm Id}_Y$.

\begin{rem}\label{rem:retract}
Assume that $\mathcal{C}$ has an initial object $\emptyset$. Then $\emptyset \rightarrow Y$ is a retract of $\emptyset \rightarrow Y'$ if and only if $Y$ is a \emph{retract} of $Y'$ in $\mathcal{C}$. The latter condition means that there exist morphisms $j\colon Y\rightarrow Y'$ and $q\colon Y'\rightarrow Y$ in $\mathcal{C}$ satisfying $q\circ j={\rm Id}_Y$.
\end{rem}

\begin{lem}
Assume that $f'\colon X'\rightarrow Y'$ is an isomorphism and that $f$ is a retract of $f'$. Then $f$ is also an isomorphism.
\end{lem}

 We point out that this result might  also be deduced from Lemmas~\ref{lem:iso-ortho} and \ref{lem:perp-retract} below.

\begin{proof}
Keep the notation  in the commutative diagram above. We observe that $f^{-1}=p\circ {f'}^{-1}\circ j$.
\end{proof}

Consider a commutative square in $\mathcal{C}$.
\[\xymatrix{
X\ar[d]_-{f}\ar[r]^-{i} & A \ar[d]^-g\\
Y\ar[r]^-{j} & B
}\]
By a \emph{lifting} for the square, we mean  a morphism $h\colon Y\rightarrow A$ making two triangles commute.
\[\xymatrix{
X\ar[d]_-{f}\ar[r]^-{i} & A \ar[d]^-g\\
Y\ar@{.>}[ru]|{h}\ar[r]^-{j} & B
}\]

\begin{defn}
Let $f\colon X\rightarrow Y$ and $g\colon A\rightarrow B$ be two morphisms in $\mathcal{C}$. We say that the ordered pair $(f, g)$ is \emph{orthogonal}, denoted by $f\perp g$, provided that any commutative diagram
\[\xymatrix{
X\ar[d]_-{f}\ar[r]^-{} & A \ar[d]^-g\\
Y\ar[r]^-{} & B
}\]
admits a lifting.
\end{defn}

We have a characterization of isomorphisms using the orthogonality.

\begin{lem}\label{lem:iso-ortho}
Let $f\colon X\rightarrow Y$ be a morphism in $\mathcal{C}$. Then the following statements are equivalent:
\begin{enumerate}
\item[(1)] $f$ is an isomorphism;
\item[(2)] $f\perp g$ for any morphism $g$;
\item[(3)] $g\perp f$ for any morphism $g$;
\item[(4)] $f\perp f$.
\end{enumerate}
\end{lem}

\begin{proof}
We only prove ``$(4)\Rightarrow (1)$". It suffices to observe that a lifting of the following diagram
\[\xymatrix{
X\ar[d]_-f \ar@{=}[r] & X \ar[d]^-f\\
Y\ar@{=}[r] & Y
}\]
is necessarily the inverse of $f$. Here, the equality signs mean the corresponding identity morphisms.
\end{proof}

For a class $\mathcal{S}$ of morphisms in $\mathcal{C}$, we define
$$^\perp\mathcal{S}=\{f\in {\rm Mor}(\mathcal{C})\; |\; f\perp g \mbox{ for any }g \in \mathcal{S}\}.$$
Dually, one defines the class $\mathcal{S}^\perp$ of morphisms. By Lemma~\ref{lem:iso-ortho}, both $^\perp\mathcal{S}$ and $\mathcal{S}^\perp$ contain all isomorphisms. For another class $\mathcal{X}$ of morphisms, we write $\mathcal{X}\perp \mathcal{S}$ if $\mathcal{X}\subseteq {^\perp \mathcal{S}}$, or equivalently, $\mathcal{S}\subseteq {\mathcal{X}^\perp}$.

The following two results are easy.

\begin{lem}\label{lem:perp-retract}
The class $^\perp \mathcal{S}$ is closed under retracts, that is, if $f$ is a retract of some element  $f'\in {^\perp\mathcal{S}}$, then we have $f\in {^\perp\mathcal{S}}$. \hfill $\square$
\end{lem}

\begin{lem}\label{lem:perp-retract-comp}
The class $^\perp \mathcal{S}$ is closed under compositions, that is, if $f\colon X\rightarrow Y$  and $g\colon Y\rightarrow Z$ lie in $^\perp \mathcal{S}$ , then we have $g\circ f\in {^\perp\mathcal{S}}$. \hfill $\square$
\end{lem}

\subsection{The small object argument}

In what follows, we assume that $\mathcal{C}$ has small colimits. In particular, it has pushouts and infinite coproducts.

\begin{lem}\label{lem:pushout}
The class $^\perp \mathcal{S}$ is closed under pushouts, that is, for any pushout diagram
\[\xymatrix{
X\ar[d]_-f \ar[r]& X' \ar[d]^-{f'}\\
Y\ar[r] & Y
}\]
with $f\in {^\perp \mathcal{S}}$, we have $f'\in {^\perp \mathcal{S}}$. \hfill $\square$
\end{lem}

\begin{lem}\label{lem:coproduct}
The class $^\perp \mathcal{S}$ is closed under infinite coproducts, that is, for any set-indexed family $\{f_i\}_{i\in \Lambda}$ of morphisms in $^\perp \mathcal{S}$, we have $\coprod_{i\in \Lambda} f_i\in {^\perp \mathcal{S}}$. \hfill $\square$
\end{lem}

Consider a sequence of morphisms in $\mathcal{C}$.
$$X_0\longrightarrow X_1 \longrightarrow X_2\longrightarrow \cdots$$
The induced morphism
$$X_0\longrightarrow {\rm colim}_{i\geq 0} \; X_i$$
 is called the \emph{transfinite composition} of the sequence; it is unique up to a unique isomorphism. We mention that any finite composition might be realized as a transfinite composition.

\begin{lem}\label{lem:trans-comp}
The class $^\perp \mathcal{S}$ is closed under transfinite compositions, that is, for any sequence above with each $X_i\rightarrow X_{i+1}$ in $^\perp \mathcal{S}$, the transfinite composition $X_0\rightarrow {\rm colim}_{i\geq 0} \; X_i$ lies in $^\perp \mathcal{S}$. \hfill $\square$
\end{lem}

\begin{defn}
Let $\mathcal{S}$ be a class of morphisms in $\mathcal{C}$. We denote by $\mathcal{S}\mbox{-cell}$ the class of morphisms formed by transfinite compositions of pushouts of coproducts of elements in $\mathcal{S}$. In other words, a morphism in $\mathcal{S}\mbox{-cell}$ is the transfinite composition of a sequence,
$$X_0\stackrel{\phi_0}\longrightarrow X_1 \stackrel{\phi_1}\longrightarrow X_2\longrightarrow \cdots$$
where each $\phi_i$ fits into a pushout diagram
\[
\xymatrix{\coprod_{x\in \Lambda_i} A_x \ar[d]_-{\coprod_{x\in \Lambda_i} f_x} \ar[r]& X_i\ar[d]^-{\phi_i}\\
\coprod_{x\in \Lambda_i} B_x \ar[r]& X_{i+1}
}\]
with each $f_x\colon A_x\rightarrow B_x\in \mathcal{S}$.
\end{defn}

\begin{prop}\label{prop:cell}
Assume that $f$ is a retract of some morphism in $\mathcal{S}\mbox{-{\rm cell}}$. Then we have $f\in {^\perp(\mathcal{S}^\perp)}$.
\end{prop}

\begin{proof}
We observe $\mathcal{S}\subseteq {^\perp(\mathcal{S}^\perp)}$. Now, we apply Lemmas~\ref{lem:coproduct}, \ref{lem:pushout}, \ref{lem:trans-comp}, and  \ref{lem:perp-retract}.
\end{proof}

\begin{defn}
An object $A$ is \emph{sequentially small}, if for any sequence $X_0\rightarrow X_1 \rightarrow X_2\rightarrow \cdots $ of morphisms in $\mathcal{C}$, the canonical map
$${\rm colim}_{i\geq 0}\; \mathcal{C}(A, X_i)\longrightarrow \mathcal{C}(A, {\rm colim}_{i\geq 0} \; X_i)$$
is a bijection.
\end{defn}

The following fundamental result is due to \cite[Lemma~II.3.3]{Qui}. For a more general version using large ordinals, we refer to \cite[Section~10.5]{Hir}.

\begin{thm}\label{thm:Quillen}{\rm (The small object argument)}
Assume that $\mathcal{S}$ is a set of  morphisms whose domains are sequentially small. Then any morphism $f$ admits a factorization $f=p\circ i$ with $i\in \mathcal{S}\mbox{-}{\rm cell}$ and $p\in \mathcal{S}^\perp$.
\end{thm}

\begin{proof}
We will factorize  $f\colon X\rightarrow Y$ as a composition  $X\rightarrow E_\infty \rightarrow Y$. Set $E_0=X$, and $f_0=f\colon E_0\rightarrow Y$. Assume that $f_i\colon E_i\rightarrow Y$ is defined. Consider the following index set.
\[
\xymatrix@R=2pt@C=8pt{& A_s \ar[dd]_-{g_s}\ar[r]^-{a_s} & E_i \ar[dd]^-{f_i}\\
S_i=\{ s\colon & & & \;|\; s \mbox{ any commutative diagram with } g_s\in \mathcal{S} \}\\
                     & B_s \ar[r]^-{b_s} & Y
                     }\]
Here, we use the assumption that $\mathcal{S}$ is a \emph{set}. Consider the following pushout diagram.
\[
\xymatrix{\coprod_{s\in S_i} A_s  \ar[d]_-{\coprod_{s\in S_i}g_s} \ar[rr]^-{\sum_{s\in S_i}a_s} && E_i\ar@{.>}[d]^-{\phi_i}\\
\coprod_{s\in S_i} B_s\ar@{.>}[rr]^-{\sum_{s\in S_i}c_s} && E_{i+1}
}\]
There is a unique morphism $f_{i+1}\colon E_{i+1}\rightarrow Y$ satisfying
$$f_{i+1}\circ\phi_i=f_i \mbox{ and } f_{i+1}\circ \sum_{s\in S_i} c_s=\sum_{s\in S_i}b_s.$$
 By induction, we obtain a sequence.
$$X=E_0\stackrel{\phi_0}\longrightarrow E_1\stackrel{\phi_1}\longrightarrow E_2 \longrightarrow \cdots $$
We define $E_\infty={\rm colim}_{i\geq 0} \; E_i$ and set $X\rightarrow E_\infty$ to be the corresponding transfinite composition. We have the induced morphism $E_\infty \rightarrow Y$ by $f_i$'s.

To verify that $E_\infty\rightarrow Y$ lies in $\mathcal{S}^\perp$, one uses the smallness assumption on the domains of elements in $\mathcal{S}$. We omit the routine verification.
\end{proof}

We have a partial converse of Proposition~\ref{prop:cell}.

\begin{cor}\label{cor:retract}
Assume that $\mathcal{S}$ is a set of  morphisms whose domains are sequentially small. Then any morphism in ${^\perp(\mathcal{S}^\perp)}$ is a retract of some morphism in $\mathcal{S}\mbox{-{\rm cell}}$.
\end{cor}

\begin{proof}
Take any $f\in {^\perp(\mathcal{S}^\perp)}$. By Theorem~\ref{thm:Quillen}, we have a commutative square
\[
\xymatrix{
X\ar[d]_-{f} \ar[r]^-i & T\ar[d]^-p\\
Y\ar@{=}[r] & Y
}\]
satisfying $p\in \mathcal{S}^\perp$ and $i\in \mathcal{S}\mbox{-cell}$. By the fact that $f\in {^\perp(\mathcal{S}^\perp)}$ and $p\in \mathcal{S}^\perp$, we have a lifting $h\colon Y\rightarrow T$ for the square. Then we have the following commutative diagram.
\[
\xymatrix{
X\ar[d]_-{f}\ar@{=}[r] & X\ar[d]^-{i} \ar@{=}[r]  & X \ar[d]^-{f}\\
Y\ar[r]^-h &T \ar[r]^-p & Y
}\]
It follows that $f$ is a retract of $i$, as required.
\end{proof}

\section{Model categories}

Throughout this section, $\mathcal{C}$ is a category with finite colimits and limits. The initial object is denoted by $\emptyset$, and the terminal object is denoted by $\ast$. We will follow \cite{DS} closely.

\begin{defn}
A \emph{model structure} on $\mathcal{C}$ is a triple $(\mathcal{C}of, \mathcal{W}e, \mathcal{F}ib)$ consisting of three classes of morphisms, which satisfy the following axioms.
\begin{enumerate}
\item[(MC1)] Any of $\mathcal{F}ib, \mathcal{W}e, \mathcal{C}of$ is closed under compositions, and contains all the identity morphisms.
\item[(MC2)]  Any of $\mathcal{F}ib, \mathcal{W}e, \mathcal{C}of$ is closed under retracts.
\item[(MC3)] The two-out-of-three property: if $f$ and $g$ are morphisms such that $g\circ f$ is defined and if two of $f$, $g$ and $g\circ f$ lie in $\mathcal{W}e$, then so does the third.
\item[(MC4)] The lifting axiom: we have $\mathcal{C}of \perp (\mathcal{W}e\cap\mathcal{F}ib)$ and $(\mathcal{C}of\cap \mathcal{W}e) \perp \mathcal{F}ib$.
\item[(MC5)]  The factorization axiom: any morphism $f$ can be factorized in two ways: (i) $f=p\circ i$ with $i\in \mathcal{C}of$ and $p\in (\mathcal{W}e\cap\mathcal{F}ib)$, and (ii)  $f=q\circ j$ with $j\in (\mathcal{C}of\cap \mathcal{W}e)$ and $q\in \mathcal{F}ib$.
\end{enumerate}
By a \emph{model category}, we mean a category with a model structure.
\end{defn}

\begin{rem}
The above model category coincides with the \emph{closed} model category in \cite[Chapter I, \S 5]{Qui}.  As mentioned in \cite[Chapter I, Introduction]{Qui}, the term ``model category" is  short for ``a category of models for a homotopy theory".
\end{rem}

Morphisms in $\mathcal{C}of$ (resp.,  $\mathcal{W}e$, $\mathcal{F}ib$) are called \emph{cofibrations} (resp., \emph{weak equivalences}, \emph{fibrations}), which will be often written as $\hookrightarrow $ (resp., $\stackrel{\sim}\rightarrow$, $\twoheadrightarrow$). Morphisms in $\mathcal{C}of\cap \mathcal{W}e$ (resp., $\mathcal{W}e\cap \mathcal{F}ib$) are called \emph{acyclic cofibrations} (resp., \emph{acyclic fibrations}), which will be written as $\stackrel{\sim}\hookrightarrow$ (resp., $\stackrel{\sim}\twoheadrightarrow$). For example, the two factorizations in (MC5) are illustrated in the following commutative diagram.
\[
\xymatrix{
& Y' \ar@{>>}[dr]^-{\stackrel{p}{\sim}}\\
X \ar[rr]|{f}  \ar@{^{(}->}[dr]_-{\stackrel{\sim}{j}} \ar@{^{(}->}[ur]^-{i} && Y\\
& X' \ar@{>>}[ur]_-{q}
}\]

An object $A$ is called \emph{cofibrant}, if the unique morphism $\emptyset\rightarrow A$ is a cofibration. Dually, an object $X$ is called \emph{fibrant}, if the unique morphism $X\rightarrow \ast$ is a fibration.

\begin{prop}\label{prop:model-or}
Assume that $(\mathcal{C}of, \mathcal{W}e, \mathcal{F}ib)$ is a model structure. Then the following two statements hold.
\begin{enumerate}
\item $\mathcal{C}of={^\perp(\mathcal{W}e\cap \mathcal{F}ib)}$ and $\mathcal{C}of^\perp=\mathcal{W}e\cap \mathcal{F}ib$.
\item $\mathcal{C}of \cap \mathcal{W}e={^\perp \mathcal{F}ib}$ and $(\mathcal{C}of \cap \mathcal{W}e)^\perp=\mathcal{F}ib$.
\end{enumerate}
\end{prop}

\begin{rem}
By combining Proposition~\ref{prop:model-or} with the results in Section 1, we infer that both $\mathcal{C}of$ and $\mathcal{C}of \cap \mathcal{W}e$ are closed under pushouts, finite coproducts and compositions. It follows that any finite coproduct of cofibrant objects is cofibrant.
\end{rem}

For an object $A$, we consider the \emph{codiagonal morphism}
$${\rm Id}_A+{\rm Id}_A\colon A\coprod A\longrightarrow A.$$

\begin{defn}
A \emph{cylinder object} of $A$ is an object $A\wedge I$ of $\mathcal{C}$ together with a diagram
$$A\coprod A \xrightarrow{i_0+i_1} A\wedge I \stackrel{\sim}\longrightarrow A$$
which factors ${\rm Id}_A+{\rm Id}_A$. The cylinder object is \emph{good} if $i_0+i_1$ is a cofibration. It is \emph{very good} if in addition the weak equivalence $A\wedge I \stackrel{\sim}\rightarrow A$ is a fibration.
\end{defn}

\begin{rem}\label{rem:cylinder}
(1) Denote the above weak equivalence by $w$.  Then we have $w\circ i_0={\rm Id}_A=w\circ i_1$. It follows from (MC3) that  both $i_0$ and $i_1$ are weak equivalences.

(2) A cylinder object is really a diagram, not just a single object. By (MC5), the object $A$ has at least one very good cylinder object.
\end{rem}

\begin{defn}
Two morphisms $f, g\colon A\rightarrow X$ in $\mathcal{C}$ are \emph{left-homotopic}, written as $f\stackrel{l}\sim g$,  provided that there exist a cylinder object
$$A\coprod A \xrightarrow{i_0+i_1} A\wedge I \stackrel{\sim}\longrightarrow A$$
of $A$ such that $f+g\colon A\coprod A \rightarrow X$ factors through $i_0+i_1$. In other words, there exists a morphism $H\colon A\wedge I \rightarrow X$, called a \emph{left-homotopy} from $f$ to $g$, such that $H\circ i_0=f$ and $H\circ i_1=g$.

The left-homotopy $H$ is \emph{good} (resp., \emph{very good}), if the cylinder object $A\wedge I$ is good (resp., very good).
\end{defn}

\begin{rem}
(1) Assume that $f\stackrel{l}\sim g$. Then $f\in \mathcal{W}e$ if and only if $g\in \mathcal{W}e$.

(2) Assume that $f\stackrel{l}\sim g$. By applying (MC5) to $i_0+i_1$, we infer that there exists always a good left-homotopy from $f$ to $g$.
\end{rem}

Dually, we take an object $X$ and consider the \emph{diagonal morphism}.
$$\begin{pmatrix}
{\rm Id}_X\\
{\rm Id}_X
\end{pmatrix}\colon X\longrightarrow X\times X$$

\begin{defn}
A \emph{path object} for $X$ is an object $X^I$ in $\mathcal{C}$ together with a diagram,
$$X \stackrel{\sim}\longrightarrow X^I \xrightarrow{\begin{pmatrix}
p_0\\
p_1
\end{pmatrix}}  X\times X$$
which factors the diagonal morphism. The path object $X^I$ is \emph{good} if $\begin{pmatrix}
p_0\\
p_1
\end{pmatrix}$ is a fibration; it is \emph{very good} if in addition the weak equivalence $X\stackrel{\sim}\rightarrow X^I$ is a cofibration.
\end{defn}

The following is dual to Remark~\ref{rem:cylinder}.

\begin{rem}
Both the morphisms $p_0$ and $p_1$ are  weak equivalences.  By (MC5), the object $X$ has at least one very good path object.
\end{rem}

\begin{defn}
Two morphisms $f, g\colon A\rightarrow X$ in $\mathcal{C}$ are \emph{right-homotopic}, written as $f\stackrel{r}\sim g$,  provided that there exist a path object
$$X \stackrel{\sim}\longrightarrow X^I \xrightarrow{\begin{pmatrix}
p_0\\
p_1
\end{pmatrix}}  X\times X$$
of $X$ such that $\begin{pmatrix} f\\ g \end{pmatrix} \colon A \rightarrow X\times X$ factors through $\begin{pmatrix}p_0\\ p_1\end{pmatrix}$. In other words, there exists a morphism $K\colon A \rightarrow X^I$, called a \emph{right-homotopy} from $f$ to $g$, such that $p_0\circ K=f$ and $p_1\circ K=g$.

The right-homotopy $K$ is \emph{good} (resp., \emph{very good}), if the path object $X^I$ is good (resp., very good).
\end{defn}

\begin{rem}
Assume that $f\stackrel{r}\sim g$. There exists a good right-homotopy from $f$ to $g$. Moreover, $f$ is a weak equivalence if and only if so is $g$.
\end{rem}

\begin{prop}\label{prop:cof-fib}
Let $A$ and $X$ be two objects in $\mathcal{C}$. Then the following statements hold.
\begin{enumerate}
\item  If $A$ is cofibrant, then $\stackrel{l}\sim$ is an equivalence relation on $\mathcal{C}(A, X)$.
\item If $X$ is fibrant,  then $\stackrel{r}\sim$ is an equivalence relation on $\mathcal{C}(A, X)$.
\item If $A$ is cofibrant and $X$ is fibrant, then the two equivalence relations $\stackrel{l}\sim$ and $\stackrel{r}\sim$  on $\mathcal{C}(A, X)$ coincide.
\end{enumerate}
\end{prop}

In the situation of Proposition~\ref{prop:cof-fib}(3), we will abbreviate $\stackrel{l}\sim$ and $\stackrel{r}\sim$  as $\sim$.  We denote by $\pi(A, X)$ the set of equivalence classes of $\mathcal{C}(A, X)$ under $\sim$.

\begin{rem}\label{rem:homo-fixed}
Assume that  $A$ is cofibrant and $X$ is fibrant. We \emph{fix} a good cylinder object $A\wedge I$ for $A$ and a good path object $X^I$ for $X$. Let $f, g\colon A\rightarrow X$ be two morphisms. Then $f\sim g$ if and only if $f\stackrel{l}\sim g$ via the fixed cylinder object $A\wedge I$, if and only if $f\stackrel{r}\sim g$ via the fixed path object $X^I$.
\end{rem}

Denote by $\mathcal{C}_{\rm cf}$ the  full subcategory of $\mathcal{C}$ formed by objects which are both cofibrant and fibrant. The homotopy $\sim$ yields an equivalence relation on $\mathcal{C}_{\rm cf}$. We form the factor category ${\mathcal{C}_{\rm cf}/{\sim}}$.

\begin{defn}
Let $\mathcal{C}$ be a model category. The \emph{homotopy category} ${\rm Ho}(\mathcal{C})$ is defined to the localization $\mathcal{C}[\mathcal{W}e^{-1}]$ of $\mathcal{C}$ with respect to all weak equivalences.
\end{defn}

The following is the  main result of \cite[Chapter I, \S 1]{Qui}.

\begin{thm}\label{thm:homo-cat}{\rm (Quillen)}
 Let $\mathcal{C}$ be a model category. Then the homotopy category ${\rm Ho}(\mathcal{C})$ exists and the canonical functor
 $${\mathcal{C}_{\rm cf}/{\sim}} \longrightarrow {\rm Ho}(\mathcal{C}), \; A\mapsto A$$
  is an equivalence.
\end{thm}

\section{The category of small categories}

In this section, we recall the natural model structure on the category of small categories.
We denote by ${\rm Cat}$ the category of small categories. We follow the treatment in \cite{Rezk}.

\subsection{Quivers and (co)limits}

Let $\mathcal{C}$ be a small category and $\mathcal{D}$ be any category. We denote by $\mathcal{H}om(\mathcal{C}, \mathcal{D})$ the \emph{functor category} from $\mathcal{C}$ to $\mathcal{D}$, whose objects are functors from $\mathcal{C}$ to $\mathcal{D}$, and whose morphisms are  natural transformations between the functors.

We denote by ${\rm Set}$ the category of sets. For each $n\geq 1$, we define a category $\mathcal{K}_n$ such that ${\rm Obj} (\mathcal{K}_n)=\{0,1\}$ and ${\rm Mor}(\mathcal{K}_n)=\{{\rm Id}_0, {\rm Id}_1, \alpha_1, \cdots, \alpha_n\}$, where all the morphisms $\alpha_i$ are from $1$ to $0$. For example, we visualize $\mathcal{K}_2$ as the following picture.
\[\xymatrix{
1 \ar@<+.7ex>[rr]^-{\alpha_1} \ar@<-.7ex>[rr]_-{\alpha_2}  &&¡¡0
}\]
We denote by $\mathcal{K}_0$ the discrete category with two objects $0$ and $1$.

Recall that a \emph{quiver} $Q=(Q_0, Q_1;s, t)$ consists of a set $Q_0$ of vertices, a set $Q_1$ of arrows and two maps $s, t\colon Q_1\rightarrow Q_0$ which assign to any arrow $\alpha$ its starting vertex $s(\alpha)$ and terminating vertex $t(\alpha)$. A morphism $\phi=(\phi_0, \phi_1)\colon Q\rightarrow Q'$ between quivers is given by two maps $\phi_0\colon Q_0\rightarrow Q'_0$ and $\phi_1\colon Q_1\rightarrow Q'_1$ satisfying
$$\phi_0 \circ s=s'\circ \phi_1 \mbox{ and } \phi_0 \circ t=t'\circ \phi_1.$$
 The above condition means that for any arrow $\alpha$ in $Q$, we have $$s'(\phi_1(\alpha))=\phi_0(s(\alpha)) \mbox{ and } t'(\phi_1(\alpha))=\phi_0(t(\alpha)).$$
 The composition of morphisms is defined naturally. We denote by ${\rm Quiv}$ the category of quivers.

We observe that any quiver $Q$ gives rise to a functor $Q\colon \mathcal{K}_2\rightarrow {\rm Set}$ such that $Q(1)=Q_1$, $Q(0)=Q_0$, $Q(\alpha_1)=s$ and $Q(\alpha_2)=t$.  This observation yields the following result.

\begin{prop}
The above correspondence gives rise to an isomorphism of categories: ${\rm Quiv}\simeq \mathcal{H}om(\mathcal{K}_2, {\rm Set})$. \hfill $\square$
\end{prop}

 Each small category $\mathcal{C}$ has its \emph{underlying quiver}
$$U(\mathcal{C})=({\rm Obj}(\mathcal{C}), {\rm Mor}(\mathcal{C}); s, t),$$
where the maps $s, t\colon {\rm Mor}(\mathcal{C})\rightarrow {\rm Obj}(\mathcal{C})$ describe the domains and codomains of morphisms, respectively. Consequently, we have the forgetful functor
$$U\colon {\rm Cat}\longrightarrow {\rm Quiv}.$$

Let $Q$ be a quiver. A \emph{path} of length $n$ is a sequence $p=\alpha_n\cdots \alpha_2\alpha_1$ of arrows satisfying $t(\alpha_i)=s(\alpha_{i+1})$ for each $1\leq i\leq n-1$. We set $s(p)=s(\alpha_1)$ and $t(p)=t(\alpha_n)$. Here, we use the convention of concatenating paths from right to left. A path of length one is nothing but an arrow. To  each vertex $i$, we associate a trivial path $e_i$ of length zero.

The \emph{path category} $\mathcal{P}(Q)$ is defined such that ${\rm Obj}(\mathcal{P}(Q))=Q_0$ and $\mathcal{P}(Q)(i, j)$ is the set of all paths from $i$ to $j$; the composition is given by  concatenation of paths. This construction gives rise to a functor
$$\mathcal{P}\colon {\rm Quiv}\longrightarrow {\rm Cat}.$$

\begin{prop}
Keep the notation as above. Then we have an adjoint pair $(\mathcal{P}, U)$.
\end{prop}

\begin{proof}
The bijection
$${\rm Cat}(\mathcal{P}(Q), \mathcal{C})\simeq {\rm Quiv}(Q, U(\mathcal{C}))$$
sends a functor $F$ to a quiver morphism $\phi=(\phi_0, \phi_1)$, where $\phi_0$ is given by the action of $F$ on objects and $\phi_1$ is given by the action of $F$ on $Q_1$.
\end{proof}

Consider  the category ${\rm Cat}$ of small categories. We observe that the empty category $\emptyset$ is the initial object, and that the \emph{unit category} $\mathbf{1}=\{*\}$ consisting of a single object and a single morphism is the terminal object.

Let $\Lambda$ be a small index category and $\mathcal{C}\colon \Lambda\rightarrow {\rm Cat}$ be a functor, that is a $\Lambda$-diagram in ${\rm Cat}$. To simplify the notation, we set $\Lambda_0={\rm Obj}(\Lambda)$ and $\Lambda_1={\rm Mor}(\Lambda)$;  for $i\in \Lambda_0$, we write $\mathcal{C}(i)$ as $\mathcal{C}_i$, which is a category; for $\alpha\colon i \rightarrow j\in \Lambda_1$, we write $\mathcal{C}(\alpha)$ as $\mathcal{C}_\alpha$, which is a functor from $\mathcal{C}_i$ to $\mathcal{C}_j$.

\begin{prop}
The category ${\rm Cat}$ has small limits and colimits.
\end{prop}

\begin{proof}
Let $\mathcal{C}\colon \Lambda\rightarrow {\rm Cat}$ be a $\Lambda$-diagram in ${\rm Cat}$. We will describe its limit and colimit.

(1) The limit category $\mathcal{L}={\rm lim}_{i\in \Lambda_0}\; \mathcal{C}$ is given as follows:
\begin{align*}
{\rm Obj}(\mathcal{L}) &={\rm lim}_{i\in \Lambda_0} \; {\rm Obj}(\mathcal{C}_i)\\
                    &=\{ (X_i)_{i\in \Lambda_0}\; |\; X_i\in {\rm Obj}(\mathcal{C}_i); \mbox{ for any } \alpha\colon i\rightarrow j \mbox{ in } \Lambda_1, \; \mathcal{C}_\alpha(X_i)=X_j\};
\end{align*}
the Hom set is given by
\begin{align*}
\mathcal{L}((X_i)_{i\in \Lambda_0}, & (Y_i)_{i\in \Lambda_0})  = {\rm lim}_{i\in \Lambda_0} \; \mathcal{C}_i(X_i, Y_i)\\
&=\{(f_i)_{i\in \Lambda_0}\; |\; f_i\in \mathcal{C}_i(X_i, Y_i); \mbox{ for any } \alpha\colon i\rightarrow j \mbox{ in } \Lambda_1, \; \mathcal{C}_\alpha(f_i)=f_j \}.
\end{align*}
The composition of morphisms in $\mathcal{L}$ is defined naturally.

(2) The colimit category $\mathcal{CL}$ is given  such that
$${\rm Obj}(\mathcal{CL})={\rm colim}_{i\in \Lambda_0}\; {\rm Obj}(\mathcal{C}_i)=\bigsqcup_{i\in \Lambda_0} {\rm Obj}(\mathcal{C}_i)/{\sim},$$
where the equivalence relation $\sim$ is generated by the following set
$$\{(X_i, X_j)\; |\; X_i\in {\rm Obj}(\mathcal{C}_i), X_j \in {\rm Obj}(\mathcal{C}_j),\;  \exists \; \alpha\colon i\rightarrow j \mbox{ satisfying } \mathcal{C}_\alpha(X_i)=X_j\}.$$
Therefore, an object in $\mathcal{CL}$ is given by an equivalence class $[X_i]$ for some object $X_i$ in $\mathcal{C}_i$.  Consider the following quiver
$$Q=({\rm Obj}(\mathcal{CL}), \bigsqcup_{i\in \Lambda_0} {\rm Mor}(\mathcal{C}_i); s, t),$$
where for any morphism $f\colon X_i\rightarrow Y_i$ in $\mathcal{C}_i$, we define $s(f)=[X_i]$ and $t(f)=[Y_i]$.

Consider the path category  $\mathcal{P}=\mathcal{P}(Q)$. Set $\mathcal{R}$ to be the equivalence relation on $\mathcal{P}$ generated by the following elements:
\begin{enumerate}
\item[(1)] $(g_i\circ f_i, g_if_i)$ for any composable morphisms $f_i$ and $g_i$ in $\mathcal{C}_i$;
\item[(2)] $({\rm Id}_{X_i}, e_{[X_i]})$ for any object $X_i$ in $\mathcal{C}_i$;
\item[(3)] $(h, \mathcal{C}_\alpha(h))$ for any $\alpha\colon i\rightarrow j$ in $\Lambda_1$ and any morphism $h$ in $\mathcal{C}_i$.
\end{enumerate}
Here, $g_i\circ f_i$ means the composite morphism in $\mathcal{C}_i$ and thus is an arrow in $Q$,  and $g_if_i$ is the corresponding path of length two in $Q$; ${\rm Id}_{X_i}$ is the identity morphism in $\mathcal{C}_i$ and thus is an arrow in $Q$, and $e_{[X_i]}$ denotes the trivial path concentrated in the vertex $[X_i]$. Finally, we have $\mathcal{CL}=\mathcal{P}/\mathcal{R}$.
\end{proof}

\begin{rem}
(1) Since the forgetful functor $U$ has a left adjoint, it preserves limits. This fact indicates that the construction of the limit category is straightforward.

(2) From the above constructions, we infer that the functor
$${\rm Obj}\colon {\rm Cat}\longrightarrow {\rm Set},$$
taking the set of objects from any given small category, commutes with limits and colimits. Indeed, this functor has a left adjoint and a right adjoint.

(3) Although the colimit category is hard to describe, coproducts of categories in ${\rm Cat}$ are rather easy, which are   simply given by taking disjoint unions of objects and morphisms, respectively.

(4) Assume that the index category $\Lambda$ is filtered. The filtered colimit $\mathcal{CL}$ is much easier to describe. We have
$${\rm Obj}(\mathcal{CL})={\rm colim}_{i\in \Lambda_0}\; {\rm Obj}(\mathcal{C}_i)=\bigsqcup_{i\in \Lambda_0} {\rm Obj}(\mathcal{C}_i)/{\sim},$$
where the equivalence relation $\sim$ is given as follows: for any $X_i\in {\rm Obj}(\mathcal{C}_i)$ and $X_j\in {\rm Obj}(\mathcal{C}_j)$, $X_i\sim X_j$ if and only if there exist  $\alpha\colon i\rightarrow k$ and $\beta\colon j\rightarrow k$ in $\Lambda_1$ such that $\mathcal{C}_\alpha(X_i)=\mathcal{C}_\beta(X_j)$. Similarly, we have
$${\rm Mor}(\mathcal{CL})={\rm colim}_{i\in \Lambda_0}\; {\rm Mor}(\mathcal{C}_i)=\bigsqcup_{i\in \Lambda_0} {\rm Mor}(\mathcal{C}_i)/{\sim},$$
where for any $f_i\in {\rm Mor}(\mathcal{C}_i)$ and $f_j\in {\rm Mor}(\mathcal{C}_j)$, $f_i\sim f_j$ if and only if  there exist $\alpha\colon i\rightarrow k$ and $\beta\colon j\rightarrow k$ in $\Lambda_1$ such that $\mathcal{C}_\alpha(f_i)=\mathcal{C}_\beta(f_j)$. Moreover, if $f\colon X\rightarrow Y$ is a morphism  in $\mathcal{C}_i$, then  $[f]\colon [X]\rightarrow [Y]$ is a morphism in $\mathcal{CL}$.
\end{rem}

\begin{exm}
Denote by $\mathcal{L}$ the path category of the following quiver of $A_2$.
\[\xymatrix{ 1 \ar[rr]^\alpha && 2}\]
Consider the following $\mathcal{K}_2$-diagram in ${\rm Cat}$.
\[\xymatrix{
{\bf 1} \ar@<.7ex>[rr]^-{\iota_1} \ar@<-.7ex>[rr]_-{\iota_2}&& \mathcal{L}
}\]
Here, $\bf{1}=\{*\}$ is the unit category, and $\iota_i$ is  the unique functor sending $*$ to $i$, for $i=1,2$. The colimit of this diagram, namely the coequalizer of $\iota_1$ and $\iota_2$, is the path category of the following Jordan quiver.
\[\xymatrix{
\cdot \ar@(ru, rd)^-{\alpha}}\]
Here, the dot means the unique vertex $[1]=[2]$.
\end{exm}

\subsection{Injections, isofibrations and liftings}\label{subs:Cfl}

We begin with two observations on equivalences of categories.

\begin{lem}
Let $F$ and $F'$ be two functors. Assume that $F$ is a retract of $F'$ and that $F'$ is an equivalence of categories. Then so is $F$.
\end{lem}

\begin{proof}
Consider the commutative diagram
\begin{align}\label{diag:retract}
\xymatrix{
\mathcal{C}\ar[d]_-{F} \ar[r]^-{i} & \mathcal{C}' \ar[d]^-{F'}\ar[r]^{p} & \mathcal{C} \ar[d]^-{F}\\
\mathcal{D} \ar[r]^-{j} & \mathcal{D}' \ar[r]^{q} & \mathcal{D}
}
\end{align}
satisfying $pi={\rm Id}_\mathcal{C}$ and $qj={\rm Id}_\mathcal{D}$.  Assume that $G'$ is a quasi-inverse of $F'$. Then $G=pG'j$ is a quasi-inverse of $F$.
\end{proof}

The following fact is standard.

\begin{lem}
Consider two composable functors $F\colon \mathcal{C}\rightarrow \mathcal{D}$ and $G\colon \mathcal{D} \rightarrow \mathcal{E}$. If any two of $F$, $G$ and $GF$ are  equivalences, then all the three are equivalences. \hfill $\square$
\end{lem}

We define the \emph{interval category}  $I$ such that ${\rm Obj}(I)=\{0, 1\}$ and ${\rm Mor}(I)=\{{\rm Id}_0, {\rm Id}_1, a, a^{-1}\}$; here,  $a$ is an isomorphism from $0$ to $1$, whose inverse is  $a^{-1}$.

\begin{defn}
A functor $F\colon \mathcal{C}\rightarrow \mathcal{D}$ is called an \emph{injection} provided that $F$ is injective on objects, that is, the map ${\rm Obj}(F)\colon {\rm Obj}(\mathcal{C})\rightarrow {\rm Obj}(\mathcal{D})$ is injective. It is called an \emph{acyclic injection} if it is an injection and an equivalence.
\end{defn}

We observe that the composition of two injections is an injection.

\begin{exm}\label{exm:fib}
The unique functor $\emptyset\rightarrow \mathbf{1}$ is an injection. The obvious functors $\mathcal{K}_0\rightarrow \mathcal{K}_1$ and $\mathcal{K}_2\rightarrow \mathcal{K}_1$, which acts on objects by the identity,  are both injections. The inclusion ${\rm inc}_0\colon \mathbf{1}\rightarrow I$ sending $*$ to $0$ is an acyclic injection.
\end{exm}

\begin{lem}
Assume that $F\colon \mathcal{C}\rightarrow \mathcal{D}$ is a retract of $F'\colon \mathcal{C}'\rightarrow \mathcal{D}'$. If $F'$ is an injection, then so is $F$.
\end{lem}

\begin{proof}
Consider the commutative diagram (\ref{diag:retract}). Then ${i}$ is injective on objects. As $F'$ is injective on objects, so is $F'i=jF$. It follows that $F$ is injective on objects.
\end{proof}

The following notion is somehow subtle.

\begin{defn}
A functor $F\colon \mathcal{C}\rightarrow \mathcal{D}$ is called an \emph{isofibration} provided that for any given isomorphism $g\colon F(C)\rightarrow D$ in $\mathcal{D}$ with $C\in {\rm Obj}(\mathcal{C})$, there exists an isomorphism $h\colon C\rightarrow C'$ in $\mathcal{C}$ satisfying $F(h)=g$, in particular, $F(C')=D$. It is called an \emph{acyclic isofibration} if it is an isofibration and an equivalence.
\end{defn}

We observe that the composition of two isofibrations is an isofibration. The following result is evident.

\begin{lem}\label{lem:ac-fib}
A functor $F\colon \mathcal{C}\rightarrow \mathcal{D}$ is an acyclic isofibration if and only if it is fully faithful and surjective on objects. \hfill $\square$
\end{lem}

 Recall the inclusion functor ${\rm inc}_0\colon \mathbf{1}\rightarrow I$; see Example~\ref{exm:fib}.

\begin{prop}\label{prop:fib-perp}
A functor $F\colon \mathcal{C}\rightarrow \mathcal{D}$ is an isofibration if and only if ${\rm inc}_0\perp F$.
\end{prop}

\begin{proof}
We only prove the ``if" part. Assume that ${\rm inc}_0\perp F$. Take any isomorphism $g\in F(C)\rightarrow D$ in $\mathcal{D}$. We define a functor $G\colon I\rightarrow \mathcal{D}$ by $G(0)=F(C)$, $G(1)=D$ and $G(a)=g$. We have a functor $\iota_C\colon \mathbf{1}\rightarrow \mathcal{C}$ given by $\iota_C(*)=C$. So, we have the following commutative square  in ${\rm Cat}$.
\[\xymatrix{
\mathbf{1} \ar[d]_-{{\rm inc}_0} \ar[r]^-{\iota_C} & \mathcal{C}\ar[d]^-{F}\\
I\ar[r]^-{G}& \mathcal{D}
}\]
Therefore, there is a lifting $H\colon I\rightarrow \mathcal{C}$ of the square above. We observe that $H(a)$ is an isomorphism in $\mathcal{C}$ satisfying $FH(a)=g$, proving that $F$ is an isofibration.
\end{proof}

\begin{lem}
Assume that $F\colon \mathcal{C}\rightarrow \mathcal{D}$ is a retract of $F'\colon \mathcal{C}'\rightarrow \mathcal{D}'$. If $F'$ is an  isofibration, then so is $F$.
\end{lem}

\begin{proof}
Combine Proposition~\ref{prop:fib-perp} and the dual of Lemma~\ref{lem:perp-retract}.
\end{proof}

The following results are similar to Proposition~\ref{prop:fib-perp} with very similar proofs. We use the injections in Example~\ref{exm:fib}.

\begin{prop}\label{prop:3-ortho}
Let $F\colon \mathcal{C}\rightarrow \mathcal{D}$ be a functor. Then the following statements hold.
\begin{enumerate}
\item[(1)] The functor $F$ is surjective on objects if and only if $(\emptyset\rightarrow \mathbf{1})\perp F$.
\item[(2)] The functor $F$ is full if and only if $(\mathcal{K}_0\rightarrow \mathcal{K}_1)\perp F$.
\item[(3)] The functor $F$ is faithful if and only if $(\mathcal{K}_2\rightarrow \mathcal{K}_1)\perp F$.
\end{enumerate}
\end{prop}

\begin{cor}
A functor $F\colon \mathcal{C}\rightarrow \mathcal{D}$ is an acyclic isofibration if and only if $\{\emptyset\rightarrow \mathbf{1},  \mathcal{K}_0\rightarrow \mathcal{K}_1,\mathcal{K}_2\rightarrow \mathcal{K}_1\}\perp F$.
\end{cor}

\begin{proof}
Combine Lemma~\ref{lem:ac-fib} and Proposition~\ref{prop:3-ortho}.
\end{proof}

\subsection{The natural model structure} The following two orthogonality results are less trivial.

\begin{prop}
Let $F\colon \mathcal{C}\rightarrow \mathcal{D}$ be an acyclic injection, and $G\colon \mathcal{M}\rightarrow \mathcal{N}$ be an isofibration. Then we have $F\perp G$.
\end{prop}

\begin{proof}
Consider the following commutative square.
\[\xymatrix{
\mathcal{C}\ar[d]_-{F} \ar[r]^-i &\mathcal{M}\ar[d]^-{G}\\
\mathcal{D} \ar[r]^-j & \mathcal{N}
}\]
Take any $X\in {\rm Obj}(\mathcal{D})$. Since $F$ is dense, we fix an object $C_X$ in $\mathcal{C}$ and an isomorphism
$$u_X\colon X\longrightarrow F(C_X)$$
in $\mathcal{D}$. Moreover, if $X=F(C)$ for some object $C$, we will take $C_X=C$ and $u_X={\rm Id}_{X}$; in this case, such an object $C$ is unique, since $F$ is an injection.

Consider the isomorphism
$$j(u_X)\colon j(X)\longrightarrow jF(C_X)=Gi(C_X)$$
in $\mathcal{N}$. Since $G$ is an  isofibration, we might fix an isomorphism
$$h_X\colon H(X)\longrightarrow i(C_X)$$
in $\mathcal{M}$ such that $G(h_X)=j(u_X)$; in particular, we have $GH(X)=j(X)$. If $X=F(C)$, we will take $H(X)=i(C)$ and $h_X={\rm Id}_{i(C)}$.

We will define a lifting functor $H\colon \mathcal{D}\rightarrow \mathcal{M}$ for the square. The above argument defines its action on objects. For any morphism $f\colon X\rightarrow Y$ in $\mathcal{D}$, there is a unique morphism $C_f\colon C_X\rightarrow C_{Y}$ in $\mathcal{C}$ making the following diagram commute.
\[
\xymatrix{
X\ar[d]_-{f} \ar[r]^-{u_X} & F(C_X)\ar@{.>}[d]^-{F(C_f)}\\
Y\ar[r]^-{u_Y} & F(C_Y)
}\]
Here, we use the fully-faithfulness of $F$. Now, we define
$$H(f)=h_Y^{-1}\circ i(C_f)\circ h_X\colon H(X)\longrightarrow H(Y).$$
 This defines the required lifting functor $H$.
\end{proof}

\begin{prop}
Let $F\colon \mathcal{C}\rightarrow \mathcal{D}$ be an injection, and $G\colon \mathcal{M}\rightarrow \mathcal{N}$ be an acyclic isofibration. Then we have $F\perp G$.
\end{prop}

\begin{proof}
Consider the following commutative square.
\[\xymatrix{
\mathcal{C}\ar[d]_-{F} \ar[r]^-i &\mathcal{M}\ar[d]^-{G}\\
\mathcal{D} \ar[r]^-j & \mathcal{N}
}\]
We will construct a lifting functor $H\colon \mathcal{D}\rightarrow \mathcal{M}$. Take an object $X$ in $\mathcal{D}$. By Lemma~\ref{lem:ac-fib}, $G$ is surjective on objects. We will fix an object $H(X)$ in $\mathcal{M}$ such that $j(X)=GH(X)$. Moreover, if $X=F(C)$ for some object $C$ in $\mathcal{C}$, we always take $H(X)=i(C)$. Here, we use the uniqueness of such $C$, as $F$ is an injection.

For any morphism $f\colon X\rightarrow Y$ in $\mathcal{D}$, by the fully-faithfulness of $G$ there is a unique morphism $H(f)\colon H(X)\rightarrow H(Y)$ in $\mathcal{M}$ such that $GH(f)=j(f)$. One verifies that $H$ is a well-defined functor and that it is a lifting for the square above.
\end{proof}

Let $F\colon \mathcal{C}\rightarrow \mathcal{D}$ be a functor. We will describe two factorizations of $F$.

We define a new category $\mathcal{D}'$ as follows: ${\rm Obj}(\mathcal{D}')={\rm Obj}(\mathcal{C})\sqcup {\rm Obj}(\mathcal{D})$; for any object $C_1, C_2\in {\rm Obj}(\mathcal{C})$ and $D_1, D_2 \in {\rm Obj}(\mathcal{D})$, the Hom-sets are given by
\begin{align*}
&\mathcal{D}'(C_1, C_2)=\mathcal{D}(F(C_1), F(C_2)),  \; \mathcal{D}'(C_1, D_1)=\mathcal{D}(F(C_1), D_1), \\
& \mathcal{D}'(D_1, C_1)=\mathcal{D}(D_1, F(C_1)),  \mbox{ and } \mathcal{D}'(D_1, D_2)=\mathcal{D}(D_1, D_2).
\end{align*}
The composition of morphisms in $\mathcal{D}'$ is induced by the one in $\mathcal{D}$. We observe that both $\mathcal{D}$ and the category $\mathcal{C}_F$ in Remark~\ref{rem:image} are full subcategories of $\mathcal{D}'$.

We have a canonical functor $j\colon \mathcal{C}\rightarrow \mathcal{D}'$ given by $j(C)=C$ and $j(f)=F(f)$. Another canonical functor $p\colon \mathcal{D}'\rightarrow \mathcal{D}$ is given such that $p(C)=F(C)$ and $p(D)=D$, whose action on morphisms is the identity.

\begin{prop}\label{prop:factor1}
The functor $F\colon \mathcal{C}\rightarrow \mathcal{D}$ admits a factorization $F=pj$. Moreover, $j\colon \mathcal{C}\rightarrow \mathcal{D}'$ is an injection and $p\colon \mathcal{D}'\rightarrow \mathcal{D}$ is an acyclic isofibration.
\end{prop}

\begin{proof}
The functor $p$ is an equivalence, whose quasi-inverse is the inclusion functor $\mathcal{D}\hookrightarrow \mathcal{D}'$. Since $p$ is clearly surjective on objects, it is an acyclic isofibration by Lemma~\ref{lem:ac-fib}.
\end{proof}

\begin{rem}
(1) Consider the functor $F_1\colon \mathcal{C}\rightarrow \mathcal{C}_F$ in Remark~\ref{rem:image}. We observe that $j$ is the composition of $F_1$ with the inclusion $\mathcal{C}_F\hookrightarrow \mathcal{D}'$.

(2) Denote by ${\rm inc}\colon \mathcal{D}\rightarrow \mathcal{D}'$ the inclusion functor. We have another factorization $F=p({\rm inc}F)$ of $F$. However, ${\rm inc}F\colon \mathcal{C} \rightarrow \mathcal{D}'$ is not an injection in general.
\end{rem}

For the given functor $F\colon \mathcal{C}\rightarrow \mathcal{D}$, we define another new category $\mathcal{C}'$ consisting of triples. More precisely, we have
$${\rm Obj}(\mathcal{C}')=\{(C, \alpha, D)\; |\; C\in {\rm Obj}(\mathcal{C}), D\in {\rm Obj}(\mathcal{D}), \alpha\colon F(C)\rightarrow D \mbox{ an isomorphism in } \mathcal{D}\}.$$
The Hom sets are given by
$$\mathcal{C}'((C, \alpha, D), (C', \alpha', D'))=\mathcal{C}(C, C'),$$
and the composition of morphisms is induced by the one in $\mathcal{C}$.

The canonical functor $\iota\colon \mathcal{C}\rightarrow \mathcal{C}'$ sends $C$ to $(C, {\rm Id}_{F(C)}, F(C))$ and acts on morphisms by the identity. The projection functor $q\colon \mathcal{C}'\rightarrow \mathcal{D}$ sends $(C, \alpha, D)$ to $D$, and a morphism $f\colon (C, \alpha, D)\rightarrow (C', \alpha', D')$ to $\alpha'\circ F(f)\circ \alpha^{-1}$.

\begin{prop}\label{prop:factor2}
The functor $F\colon \mathcal{C}\rightarrow \mathcal{D}$ admits a factorization $F=q\iota$. Moreover, $\iota\colon \mathcal{C}\rightarrow \mathcal{C}'$ is an acyclic injection and $q\colon \mathcal{C}'\rightarrow \mathcal{D}$ is an isofibration.
\end{prop}

\begin{proof}
For any isomorphism $\beta\colon q(C, \alpha, D)=D\rightarrow D'$ in $\mathcal{D}$, we have an object $(C, \beta\circ \alpha, D')$. The following morphism
$${\rm Id}_C\colon (C, \alpha, D)\longrightarrow (C, \beta\circ \alpha, D')$$
is an isomorphism, which is sent by $q$ to the given isomorphism $\beta$. This proves that $q$ is an isofibration.

 Consider the projection functors ${\rm pr}_1\colon \mathcal{C}'\rightarrow \mathcal{C}$. We observe that ${\rm pr}_1$ is a quasi-inverse of $\iota$.
\end{proof}

\begin{rem}
We have another factorization $F=(F{\rm pr}_1)\iota$. However, $F{\rm pr}_1\colon \mathcal{C}'\rightarrow \mathcal{D}$ is not an isofibration in general.
\end{rem}

\begin{exm}
Take $F$ above to be the identity functor ${\rm Id}_\mathcal{C}\colon \mathcal{C}\rightarrow \mathcal{C}$. Then the category $\mathcal{D'}$ is isomorphic to $\mathcal{C}\times I$, and $\mathcal{C}'$ is isomorphic to $\mathcal{H}om(I, \mathcal{C})$. Here, the interval category $I$ is defined in Subsection~\ref{subs:Cfl},  and $\mathcal{C}\times I=\mathcal{C}\prod I$ is their product.
\end{exm}

We summarize the results in this subsection.  We denote by $\mathcal{I}nj$ the class of all injections, by $\mathcal{E}quiv$ the class of all equivalences, and by $\mathcal{I}sof$ the class of all isofibrations.

\begin{thm}
The triple $(\mathcal{I}nj, \mathcal{E}quiv, \mathcal{I}sof)$ is a model structure on ${\rm Cat}$.
\end{thm}

This model structure is often called the \emph{natural model structure} on ${\rm Cat}$.  The above result is  well known, which might be traced back to \cite{And, Bou}. We follow the full proof given in \cite{Rezk}. We mention that there are other model structures on ${\rm Cat}$, for example, the \emph{Thomason model structure} \cite{Thom} and the \emph{Morita model structure} \cite[Section~9.2]{Bal}.

In the natural model structure, every small category is both cofibrant and fibrant. Therefore, $\stackrel{l}\sim$ coincides with $\stackrel{r}\sim$. In the next subsection, we will prove that two functors are homotopic if and only if they are naturally isomorphic; see Theorem~\ref{thm:nat-iso}.

\subsection{Natural isomorphisms}

Let $\mathcal{C}$ be a small category. We consider the \emph{codiagonal functor}
$${\rm Id}_\mathcal{C}+{\rm Id}_\mathcal{C}\colon \mathcal{C}\coprod \mathcal{C}\longrightarrow \mathcal{C}.$$
Here, $\mathcal{C}\coprod \mathcal{C}$ denotes the coproduct in ${\rm Cat}$, that is, the disjoint union.

Recall the interval category $I$ from Subsection~\ref{subs:Cfl}, which has a nontrivial isomorphism $a\colon 0\rightarrow 1$. Denote by $\mathcal{C}\times I$ the product of $\mathcal{C}$ and $I$, whose objects are of the form $(C, i)$ for $C\in {\rm Obj}(\mathcal{C})$ and $i=0, 1$. For each $i$, we have a functor $\iota_i\colon \mathcal{C}\rightarrow \mathcal{C}\times I$ given by $\iota_i(C)=(C, i)$. The projection functor ${\rm pr}\colon \mathcal{C}\times I\rightarrow \mathcal{C}$ sends $(C, i)$ to $C$.

\begin{lem}
The codiagonal functor ${\rm Id}_\mathcal{C}+{\rm Id}_\mathcal{C}$ admits the following factorization.
$$ \mathcal{C}\coprod \mathcal{C}\xrightarrow{\iota_0+\iota_1} \mathcal{C}\times I \stackrel{\rm pr}\longrightarrow \mathcal{C}$$
Moreover, $\iota_0+\iota_1$ is an injection and ${\rm pr}$ is an acyclic isofibration. Therefore, it is a very good cylinder object for $\mathcal{C}$.
\end{lem}

\begin{proof}
We just mention that both $\iota_i$ are quasi-inverses of ${\rm pr}$. In particular, $\iota_0$ and $\iota_1$ are naturally isomorphic.
\end{proof}

\begin{lem}\label{lem:nat-iso1}
Let $F, G\colon \mathcal{C}\rightarrow \mathcal{D}$ be two functors. Then $F$ and $G$ are naturally isomorphic if and only if there exists a functor $H\colon \mathcal{C}\times I\rightarrow \mathcal{D}$ satisfying $H(\iota_0+\iota_1)=(F+G)$, that is, $H\iota_0=F$ and $H\iota_1=G$.
\end{lem}

\begin{proof}
The ``if" part is easy, since $\iota_0$ and $\iota_1$ are naturally isomorphic.

For the ``only if" part, we assume that $\eta\colon F\rightarrow G$ is a natural isomorphism. The required functor $H\colon \mathcal{C}\times I\rightarrow \mathcal{D}$ is defined as follows: $H(C, 0)=F(C)$ and $H(C, 1)=G(C)$; for any morphism $f\colon C\rightarrow C'$ in $\mathcal{C}$, we have $H(f, {\rm Id}_0)=F(f)$, $H(f,a)=\eta_{C'}\circ F(f)=G(f)\circ \eta_C$, $H(f, {\rm Id}_1)=G(f)$, and $H(f, a^{-1})=\eta^{-1}_{C'}\circ G(f)=F(f)\circ \eta^{-1}_C$. In particular, we have $H({\rm Id}_C, a)=\eta_C$ and $H({\rm Id}_C, a^{-1})=\eta_C^{-1}$.
\end{proof}

For the dual consideration, we take a small category $\mathcal{D}$ and consider the \emph{diagonal functor}.
$$\begin{pmatrix}{\rm Id}_\mathcal{D} \\ {\rm Id}_\mathcal{D}  \end{pmatrix}\colon \mathcal{D} \longrightarrow \mathcal{D}\times \mathcal{D}$$

 We will identify an object in $\mathcal{H}om(I, \mathcal{D})$ with a triple $(D_0, \alpha, D_1)$ with $\alpha\colon D_0\rightarrow D_1$ an isomorphism in $\mathcal{D}$. For $i=0, 1$, the projection functor ${\rm p}_i\colon \mathcal{H}om(I, \mathcal{D})\rightarrow \mathcal{D}$ sends $(D_0, \alpha, D_1)$ to $D_i$. We have the constant functor ${\rm const} \colon \mathcal{D}\rightarrow \mathcal{H}om(I, \mathcal{D})$ sending $D$ to $(D, {\rm Id}_D, D)$.

\begin{lem}
The diagonal functor above admits the following factorization.
$$ \mathcal{D} \xrightarrow{\rm const} \mathcal{H}om(I, \mathcal{D}) \stackrel{\begin{pmatrix} p_0\\ p_1 \end{pmatrix}} \longrightarrow \mathcal{D}\times \mathcal{D}$$
Moreover, ${\rm const}$ is an acyclic injection and $\begin{pmatrix} p_0\\ p_1 \end{pmatrix}$ is an isofibration. Therefore, it is a very good path object for $\mathcal{D}$.
\end{lem}

\begin{proof}
We just mention that both $p_i$ are quasi-inverses of the constant functor. In particular, the functors $p_0$ and $p_1$ are naturally isomorphic.
\end{proof}

\begin{lem}\label{lem:nat-iso2}
Let $F, G\colon \mathcal{C}\rightarrow \mathcal{D}$ be two functors. Then $F$ and $G$ are naturally isomorphic if and only if there exists a functor $K\colon \mathcal{C} \rightarrow \mathcal{H}om(I, \mathcal{D})$ satisfying $ \begin{pmatrix} p_0\\ p_1 \end{pmatrix} K= \begin{pmatrix} F\\ G \end{pmatrix}$, that is, $p_0K=F$ and $p_1 K=G$.
\end{lem}

\begin{proof}
We only sketch the proof of the ``only if" part. Assume that $\eta\colon F\rightarrow G$ is a natural isomorphism. Then the required functor $K$ sends $C$ to the triple $(F(C), \eta_C, G(C))$.
\end{proof}

\begin{rem}
We have an adjunction isomorphism between categories:
$$\mathcal{H}om(\mathcal{C}\times I, \mathcal{D}) \simeq \mathcal{H}om(\mathcal{C}, \mathcal{H}om(I, \mathcal{D})).$$
This isomorphism sends the functor $H$ in Lemma~\ref{lem:nat-iso1} to the functor $K$ in Lemma~\ref{lem:nat-iso2}.
\end{rem}

Denote by ${\rm Ho}({\rm Cat})$ the homotopy category of ${\rm Cat}$ with respect to the natural model structure. We denote by $[{\rm Cat}]$ the category of small categories, whose morphisms are isomorphism classes of functors.

\begin{thm}\label{thm:nat-iso}
Any two functors in ${\rm Cat}$ homotopic if and only if they are naturally isomorphic. Consequently, we have an isomorphism of categories: $[{\rm Cat}]\simeq {\rm Ho}({\rm Cat})$.
\end{thm}

\begin{proof}
In view of Remark~\ref{rem:homo-fixed}, Lemma~\ref{lem:nat-iso1} or Lemma~\ref{lem:nat-iso2} implies the first statement. The second one follows from Theorem~\ref{thm:homo-cat}.
\end{proof}

\subsection{Two remarks}

Let $F\colon \mathcal{C}\rightarrow \mathcal{D}$ be a functor. The following remark allows us to view the category $\mathcal{D}'$ in Proposition~\ref{prop:factor1} as the \emph{functor cylinder} of $F$, which is a categorical analogue of the mapping cylinder in topology.

Denote by ${\rm inc}\colon \mathcal{D}\rightarrow \mathcal{D}'$ the inclusion functor. Recall the functors $j\colon \mathcal{C}\rightarrow \mathcal{D'}$ and  $p\colon \mathcal{D}'\rightarrow \mathcal{D}$ in Proposition~\ref{prop:factor1}.

\begin{rem}
(1) The functors ${\rm inc} F$ and $j$ are naturally isomorphic. By Lemma~\ref{lem:nat-iso1}, there exists a left homotopy $H\colon \mathcal{C}\times I\rightarrow \mathcal{D}'$ from ${\rm inc} F$ to $j$. More precisely, we have $H(C, 0)=F(C)$ and $H(C, 1)=C$; moreover, for any morphism $f$ in $\mathcal{C}$, all the morphisms $H(f, {\rm Id}_0)$, $H(f, {\rm Id}_1)$, $H(f, a)$ and $H(f, a^{-1})$ are represented by $F(f)$.

 (2) The following commutative diagram
\[\xymatrix{
\mathcal{C}\ar[d]_-{\iota_0} \ar[rr]^-F && \mathcal{D}\ar[d]^-{\rm inc}\\
\mathcal{C}\times I \ar[rr]^-H && \mathcal{D}'
}\]
is a pushout in ${\rm Cat}$.

(3) Using $p{\rm inc}={\rm Id}_\mathcal{D}$, we have
$$F=p{\rm inc} F=pH\iota_0.$$
We observe the following commutative diagram.
\[\xymatrix{
\mathcal{C}\times I \ar[d]_-{\rm pr} \ar[rr]^-H && \mathcal{D}' \ar[d]^-p\\
\mathcal{C} \ar[rr]^-F && \mathcal{D}
}\]
Using ${\rm pr}\iota_0={\rm pr}\iota_1$ and this commutativity, we infer that
$$pH\iota_0=F {\rm pr} \iota_0=F{\rm pr} \iota_1=pH\iota_1=pj.$$
The rightmost equality uses  the fact that $j=H\iota_1$. In summary, we have $F=pj$, recovering the factorization in Proposition~\ref{prop:factor1}.
\end{rem}

Dually, we consider the category $\mathcal{C}'$ in Proposition~\ref{prop:factor2}. We have the projection functor ${\rm pr}_1\colon \mathcal{C}'\rightarrow \mathcal{C}$, the acyclic injection $\iota\colon \mathcal{C}\rightarrow \mathcal{C}'$ and the isofibration $q\colon \mathcal{C'}\rightarrow \mathcal{D}$.
We will view $\mathcal{C'}$ as the \emph{functor cocylinder} of $F$, which is a categorical analogue of the mapping cocylinder in topology.

\begin{rem}
(1) The functors $F {\rm pr}_1$ and $q$ are naturally isomorphic. By Lemma~\ref{lem:nat-iso2},  we have a
right-homotopy $K\colon \mathcal{C'}\rightarrow \mathcal{H}om(I, \mathcal{D})$.  It sends $(C, \alpha, D)$ to $(F(C), \alpha, D)$, and a morphism $f\colon (C, \alpha, D) \rightarrow (C', \alpha', D')$ in $\mathcal{C}'$ to the morphism represented by the following commutative diagram.
\[
\xymatrix{
F(C)\ar[d]_-{f} \ar[rr]^-\alpha && D \ar[d]^-{\alpha'\circ f\circ \alpha^{-1}}\\
F(C') \ar[rr]^-{\alpha'} && D'
}\]

(2) The following commutative diagram
\[\xymatrix{
\mathcal{C}' \ar[d]_-{{\rm pr}_1} \ar[rr]^-K && \mathcal{H}om(I, \mathcal{D}) \ar[d]^-{p_0}\\
\mathcal{C} \ar[rr]^-F && \mathcal{D}
}\]
is a pullback in ${\rm Cat}$.

(3) We use ${\rm pr}_1\iota={\rm Id}_\mathcal{C}$, the following commutative diagram
\[\xymatrix{
\mathcal{C}\ar[d]_-{\iota} \ar[rr]^-F && \mathcal{D}\ar[d]^-{\rm const}\\
\mathcal{C}' \ar[rr]^-K && \mathcal{H}om(I, \mathcal{D})
}\]
and $p_0{\rm const}=p_1{\rm const}$. Then  we have
$$F=F{\rm pr}_1\iota=p_0K\iota=p_0{\rm const} F=p_1{\rm const} F=p_1 K\iota=q\iota.$$
The rightmost equality uses the fact that $p_1K=q$.  We recover the  factorization in Proposition~\ref{prop:factor2}.
\end{rem}

\section{The category of dg algebras}

Let $\mathbb{K}$ be a commutative unital ring. We will study the category ${\rm DGAlg}$ formed by  differential graded unital $\mathbb{K}$-algebras, which are abbreviated as dg algebras. We emphasize that in the dg setting, we always study homogeneous elements.

 \subsection{Cochain complexes and tensor algebras}

 We denote by $C(\mathbb{K}\mbox{-}{\rm Mod})$ the abelian category of cochain complexes. A cochain complex is usually denoted by $X=(X^n, d_X^n)_{n\in \mathbb{Z}}$, where the differentials $d_X^n\colon X^n\rightarrow X^{n+1}$ satisfy $d_X^{n+1}\circ d_X^n=0$. We denote by $B^n(X)$, $Z^n(X)$ and $H^n(X)$ the $n$-th coboundary, cocycle and cohomology of $X$, respectively. For $x\in X^n$, we say that the degree $|x|$ is $n$.

 The following well-known result will be used often.

 \begin{prop}\label{prop:surj-quas}
 Let $f\colon X\rightarrow Y$ be a cochain map. Then the following conditions are equivalent.
 \begin{enumerate}
 \item The cochain map $f$ is a surjective quasi-isomorphism.
 \item For each $n$, $Z^n(f)\colon Z^n(X)\rightarrow Z^n(Y)$ is surjective and $H^n(f)\colon H^n(X)\rightarrow H^n(Y)$ is injective.
 \item For any $n$, any $x\in X^{n+1}$ and $y\in Y^n$ satisfying $d_X(x)=0$ and $f(x)=d_Y(y)$, there exists $x'\in X^n$ such that $d_X(x')=x$ and $f(x')=y$.
 \end{enumerate}
 \end{prop}

\begin{proof}
For ``(1) $\Leftrightarrow$ (2)", we use the following two commutative diagrams with exact rows:
\[\xymatrix{
0\ar[r] & Z^n(X)\ar[d]_-{Z^n(f)} \ar[r] & X^n\ar[d]^-{f^n} \ar[r] & B^{n+1}(X)\ar[d]^-{B^{n+1}(f)} \ar[r] & 0\\
0\ar[r] & Z^n(Y) \ar[r] & Y^n \ar[r] & B^{n+1}(Y) \ar[r] & 0
}\]
and
\[\xymatrix{
0\ar[r] & B^n(X)\ar[d]_-{B^n(f)} \ar[r] & Z^n(X)\ar[d]^-{Z^n(f)} \ar[r] & H^n(X)\ar[d]^-{H^n(f)} \ar[r] & 0\\
0\ar[r] & B^n(Y) \ar[r] & Z^n(Y) \ar[r] & H^n(Y) \ar[r] & 0.
}\]

For ``(2) $\Leftrightarrow$ (3)", we only prove ``(2) $\Rightarrow$ (3)". Consider the cohomological class $[x]\in H^{n+1}(X)$. Then $H^{n+1}(f)$ vanishes on it, and the assumption implies that $[x]=0$. Therefore, there exists $x_1\in X^{n}$ satisfying $d_X(x_1)=x$. We have
$$d_Y(y)=f(x)=fd_X(x_1)=d_Y f(x_1).$$
It follows that $y-f(x_1)$ lies in $Z^n(Y)$. By the assumption, there exists $x_2\in Z^n(X)$ satisfying $f(x_2)=y-f(x_1)$. Then the required $x'$ is might be taken to be $x_1+x_2$.
\end{proof}

For each complex $X$, its \emph{suspension} $\Sigma(X)$ is defined such that $\Sigma(X)^n=X^{n+1}$ and $d_{\Sigma(X)}=-d_X$. The following ``$s$-notation" is very convenient: for any $x\in X$, the corresponding element in $\Sigma(X)$ is denoted by $sx$. Therefore, we have
\begin{center}
$|sx|=|x|-1$ and $d_{\Sigma(X)}(sx)=-sd_X(x)$.
\end{center}
In general, we can define $\Sigma^n(X)$, whose typical element is denoted by $s^nx$.

Let $f\colon X\rightarrow Y$ be a cochain map. Its \emph{mapping cone} ${\rm Cone}(f)$ is defined as follows: ${\rm Cone}(f)=Y\oplus \Sigma(X)$ as a graded module, and its differential is given by
$$d_{{\rm Cone}(f)}=\begin{pmatrix} d_Y & f\circ \xi^{-1}\\
                                      0 & d_{\Sigma(X)}\end{pmatrix}.$$
                                      Here, $\xi\colon X\rightarrow \Sigma(X)$ is the canonical map sending $x$ to $sx$.

For a cochain complex $V$, we consider the \emph{tensor dg algebra}
$$T(V)=\bigoplus_{n\geq 0} V^{\otimes n}.$$
For any dg algebra $A$, we have a natural isomorphism
\begin{align}\label{iso:tensor}
{\rm DGAlg}(T(V), A)\simeq C(\mathbb{K}\mbox{-}{\rm Mod})(V, A), \; \theta\mapsto \theta|_{V}.
\end{align}

The two dg algebras in the following examples play a central role.

\begin{exm}
Consider a stalk complex $\mathbb{K}u$ with a free generator $u$ of degree $-n$. The tensor algebra is denoted by $S(n)$, called the $n$-th \emph{sphere}. We mention that $S(n)=\mathbb{K}[u]$ is a graded polynomial algebra with $|u|=-n$ and the zero differential. The isomorphism (\ref{iso:tensor}) is specialized as follows:
$${\rm DGAlg}(S(n), A)\simeq Z^{-n}(A), \theta\mapsto \theta(u).
$$
To emphasize the generator $u$, we often use the better notation $S(u)$ for $S(n)$.
\end{exm}

\begin{exm}\label{exm:disc}
Consider the complex $V=\mathbb{K}t\oplus \mathbb{K}dt$ with two free generators $t$ and $dt$, whose degrees are given by $|t|=-n$ and $|dt|=1-n$; the  differential of $V$ is determined by $d_V(t)=dt$ and $d_V(dt)=0$. The tensor algebra $T(V)$ is denoted by $D(n)$, called the $n$-th \emph{disc}. We mention that $D(n)$ is the free algebra $\mathbb{K}\langle t, dt\rangle$ with two generators $t$ and $dt$. The isomorphism (\ref{iso:tensor}) is specialized as follows:
$${\rm DGAlg}(D(n), A)\simeq A^{-n},  \theta\mapsto \theta(t).$$
We sometimes denote the above disc by $D(t)$.
\end{exm}

\begin{exm}
Let $A$ be a dg algebra and $_AM_A$ be a dg $A$-$A$-bimodule. We denote by
$$T_A(M)=A\oplus M \oplus M\otimes_A M\oplus M\otimes_A M\otimes_A M\oplus \cdots$$
the tensor dg algebra. Its differential $d$ is uniquely determined by $d(a)=d_A(a)$ for any $a\in A$, and $d(x)=d_M(x)$ for any $x\in M$.

Consider a homomorphism  $\alpha\colon M\rightarrow A$ of dg $A$-$A$-modules of degree one. In other words, we have $|\alpha(x)|=|x|+1$ and $d_A(\alpha (x)) = -\alpha(d_M (x))$ for any $x\in M$. Moreover, we have $\alpha(a.y)=(-1)^{|a|} a.\alpha(y)$ and $\alpha(y.b)= \alpha(y).b$ for any $y\in M$ and $a, b\in A$. The \emph{defomed tensor dg algebra} $T_A(M;\alpha)$ is defined as follows: it equals $T_A(M)$ as a graded algebra, and its differential $d'$ is uniquely determined by $d'(a)=d_A(a)$ for any $a\in A$, and $d'(x)=d_M(x)+\alpha(x)$ for any $x\in M$.
\end{exm}

\subsection{Limits and colimits}

Let $\Lambda$ be a small index category. Take a $\Lambda$-diagram  $A\colon \Lambda\rightarrow {\rm DGAlg}$ of dg algebras. Therefore, for each object $i\in \Lambda_0$, $A_i$ is a dg algebra, and for each $\alpha\colon i \rightarrow j$ in $\Lambda_1$, $A_\alpha\colon A_i\rightarrow A_j$ is a dg algebra homomorphism.

We mention that $\mathbb{K}$ is the initial object in ${\rm DGAlg}$, and the zero algebra $0$ is the terminal object in ${\rm DGAlg}$.

\begin{prop}
The category ${\rm DGAlg}$ has small limits and colimits.
\end{prop}

\begin{proof}
Let  $A\colon \Lambda\rightarrow {\rm DGAlg}$ be a $\Lambda$-diagram. We will describe its limit and colimit.

The limit algebra ${\rm lim}_{i\in \Lambda_0} \; A_i$ is a dg subalgebra of the product dg algebra $\prod_{i\in \Lambda_0} A_i$:
$${\rm lim}_{i\in \Lambda_0} \; A_i=\{(a_i)_{i\in \Lambda_0} \in \prod_{i\in \Lambda_0} A_i\; |\; A_\alpha(a_i)=a_j \mbox{ for any } \alpha\colon i\rightarrow j\in\Lambda_1 \}.$$

The colimit algebra ${\rm colim}_{i\in \Lambda_0}\; A_i$ is given by the quotient dg algebra
$$T(\bigoplus_{i\in \Lambda_0} A_i)/I, $$
where the two-sided dg ideal $I$ is generated by the following elements:
\begin{enumerate}
\item $x_i\otimes x'_i-x_ix'_i$, any  elements $x_i, x'_i\in A_i$;
\item $1_{A_i}-1$, $i\in \Lambda_0$;
\item $A_\alpha(x)-x$, for any $\alpha\colon i\rightarrow j\in \Lambda_1$ and $x\in A_i$.
\end{enumerate}
Here,  $x_i\otimes x_i'$ is a tensor of length two, and $x_ix'_i$ denotes the product of $x_i$ and $x'_i$ in $A_i$, and thus is viewed as a tensor of length one; $1$ is the  unit of the tensor algebra,  and  $1_{A_i}$ is the unit of $A_i$, which is a tensor of length one.
\end{proof}

\begin{rem}
(1) The coproduct $\coprod_{i\in \Lambda_0} A_i$ of dg algebras  is often called the \emph{free product} of dg algebras, which is often denoted by $*_{i\in \Lambda_0} A_i$.

(2)  If the index category $\Lambda$ is filtered, the filtered colimit algebra ${\rm colim}_{i\in \Lambda}\; A_i$ is easy to describe:
$${\rm colim}_{i\in \Lambda}\; A_i=\bigsqcup_{i\in \Lambda_0} A_i/{\sim},$$
where the equivalence relation  $\sim$ is given such that $x_i\sim x_j$ for $x_i\in A_i$ and $x_j\in A_j$, provided that there exist $\alpha\colon i\rightarrow k$ and $\beta\colon j\rightarrow k$ in $\Lambda_1$ such that $A_\alpha(x_i)=A_\beta(x_j)$ in $A_k$.
\end{rem}

\begin{exm}
Consider the natural embedding $\iota_n\colon S(n)\rightarrow D(n+1)$ of dg algebras sending $u$ to $dt$, and the surjective homomorphism  $\pi_n\colon D(n)\rightarrow S(n)$ sending $t$ to $u$, and $dt$ to $0$. We observe that the following diagram is a pushout.
\[\xymatrix{
S(n) \ar[d]_-{\iota_n} \ar[rr] && \mathbb{K}\ar[d]\\
D(n+1) \ar[rr]^-{\pi_{n+1}} && S(n+1)
}\]
Here, the unnamed horizontal arrow sends $u$ to $0$, and the unnamed vertical arrow is the unique morphism from $\mathbb{K}$ to $S(n+1)$.
\end{exm}

The following general consideration extends the above one.

\begin{exm}
Let $f\colon X\rightarrow Y$ be a cochain map. Denote by ${\rm Cone}(f)$ the mapping cone. Let $A$ be a dg algebra and $\phi\colon T(Y)\rightarrow A$  a dg algebra homomorphism. We have the following homomorphism of degree one between dg $A$-$A$-modules.
$$\alpha\colon A\otimes \Sigma(X)\otimes A\longrightarrow A, \; a\otimes sx\otimes b\mapsto (-1)^{|a|} a\phi(f(x))b$$
Therefore, we form the deformed dg tensor algebra $T_A(A\otimes \Sigma(X)\otimes A, \alpha)$ and have the inclusion ${\rm inc}\colon A\rightarrow T_A(A\otimes \Sigma(X)\otimes A, \alpha)$. The following commutative diagram
\[
\xymatrix{
T(Y)\ar[d] \ar[rr]^\phi && A\ar[d]^-{\rm inc}\\
T({\rm Cone}(f))\ar[rr]^-{\psi} &&  T_A(A\otimes \Sigma(X)\otimes A; \alpha)
}\]
is a pushout, where the unnamed arrow is induced from $\begin{pmatrix} 1\\ 0 \end{pmatrix}\colon Y\rightarrow {\rm Cone}(f)$, and the homomorphism $\psi$ sends any element $\begin{pmatrix} y\\ sx \end{pmatrix}$ in  ${\rm Cone}(f)$ to $\phi(y)+1_A\otimes sx\otimes 1_A$ in $A\otimes \Sigma(X)\otimes A$.
\end{exm}

Let us consider two special cases.

\begin{exm}\label{exm:two-po}
Let $A$ be a dg algebra.

(1) Recall the contractible complex $V=\mathbb{K}t\oplus \mathbb{K}dt$ from Example~\ref{exm:disc}. Then we have the following pushout diagram,
\[
\xymatrix{\mathbb{K}\ar[d] \ar[rr] && A\ar[d]^-{\rm inc}\\
D(t) \ar[rr] && T_A(A\otimes V\otimes A)
}\]
where the inclusion $A\rightarrow T_A(A\otimes V\otimes A)$ is a quasi-isomorphism. We mention that $T_A(A\otimes V\otimes A)=A\ast D(t)$ is the free product.

(2) Take $x\in A$ a cocycle of degree $-n$. Then we have the following pushout diagram
\[
\xymatrix{S(n)\ar[d]_-{\iota_n} \ar[rr]^-{u\mapsto x} && A\ar[d]^-{\rm inc}\\
D(n+1) \ar[rr] && T_A(A\otimes \mathbb{K}t\otimes A; \alpha).
}\]
Here, $\alpha\colon A\otimes \mathbb{K}t\otimes A\rightarrow A$ sends $a\otimes t \otimes b$ to $(-1)^{|a|}axb$, and $T_A(A\otimes \mathbb{K}t\otimes A; \alpha)$ is the deformed dg tensor algebra. The unnamed arrow sends $t$ to $1_A\otimes t\otimes 1_A$, and $dt$ to $x$.

We often write $A\langle t; x\rangle$ for $T_A(A\otimes \mathbb{K}t\otimes A; \alpha)$. In the extension $A\rightarrow A\langle t; x\rangle$, the given cocycle $x$ in $A$ becomes a coboundary in $A\langle t; x\rangle$.
\end{exm}

\subsection{The model structure}

We denote by $\mathcal{Q}uas$ the class formed by all quasi-isomorphisms in ${\rm DGAlg}$, and by $\mathcal{S}urj$ the class formed by all surjective homomorphisms in ${\rm DGAlg}$. We define $\mathcal{C}of={^\perp(\mathcal{Q}uas\cap \mathcal{S}urj)}$, whose elements will be called \emph{cofibrations}.

The main result of this section is due to \cite[Theorem~5]{Jar}.

\begin{thm}\label{thm:Jardine}
The triple $(\mathcal{C}of, \mathcal{Q}uas, \mathcal{S}urj)$ is a model structure on ${\rm DGAlg}$.
\end{thm}

\begin{rem}
We mention that the category ${\rm DGCAlg}$ of differential graded-commutative algebras has a similar model structure \cite{BG}.
\end{rem}

Let us make some preparation.

\begin{lem}\label{lem:dgcat-model}
The following two statements hold.
\begin{enumerate}
\item $\mathcal{S}urj=\{\mathbb{K}\rightarrow D(n)\; |\;  n\in \mathbb{Z}\}^\perp$;
\item $\mathcal{Q}uas\cap \mathcal{S}urj=\{\iota_n\colon S(n)\rightarrow D(n+1)\; |\;  n\in \mathbb{Z}\}^\perp$.
\end{enumerate}
\end{lem}

\begin{proof}
(1) Use the isomorphism in Example~\ref{exm:disc}.

(2) We only prove that any surjective quasi-isomorphism $f\colon A\rightarrow B$ lies in  $\{\iota_n\; |\;  n\in \mathbb{Z}\}^\perp$. Consider the following commutative square.
\[
\xymatrix{
S(n)\ar[d]_-{\iota_n} \ar[rr]^-\theta  && A\ar[d]^-f\\
D(n+1) \ar[rr]^-{\eta}  && B
}\]
We observe that $\theta(u)=a\in Z^{-n}(A)$  and $\eta(t)=b\in B^{-n-1}$ satisfy $d_B(b)=f(a)$. In view of Proposition~\ref{prop:surj-quas}, there exists $a'\in A^{-n-1}$ such that $d_A(a')=a$ and $f(a')=b$. It gives a lifting $h\colon D(n+1)\rightarrow A$ by setting $h(t)=a'$ and $h(dt)=a$.
\end{proof}

In what follows, we set $\mathcal{I}=\{\iota_n\colon S(n)\rightarrow D(n+1)\; |\; n\in \mathbb{Z}\}$ and $\mathcal{J}=\{\mathbb{K}\rightarrow D(n)\; |\; n\in \mathbb{Z}\}$. We mention that the domains of these morphisms are all sequentially small.

\begin{lem}\label{lem:J}
We have $\mathcal{J}\mbox{-{\rm cell}}\subseteq (^\perp\mathcal{S}urj\cap \mathcal{Q}uas)\subseteq (\mathcal{C}of\cap \mathcal{Q}uas)$.
\end{lem}

\begin{proof}
In view of Example~\ref{exm:two-po}(1), we infer that $\mathcal{J}\mbox{-cell}\subseteq \mathcal{Q}uas$. The results follow from $\mathcal{J}\mbox{-cell}\subseteq {^\perp(\mathcal{J}^\perp)}={^\perp\mathcal{S}urj}\subseteq \mathcal{C}of$.
\end{proof}

\begin{lem}\label{lem:dgalg-cof}
A homomorphism $f\colon A\rightarrow B$ is a cofibration if and only if it is a retract of an element in $\mathcal{I}\mbox{-}{\rm cell}$.
\end{lem}

\begin{proof}
By Lemma~\ref{lem:dgcat-model}(2), we have $\mathcal{C}of={^\perp(\mathcal{I}^\perp)}$. Then the result follows from Corollary~\ref{cor:retract}.
\end{proof}

\noindent \emph{Proof of Theorem~\ref{thm:Jardine}.}\; The axioms (MC1)-(MC3) are easy to verify.

For (MC5), we apply the small object argument to $\mathcal{I}$. Let $f\colon A\rightarrow B$ be a homomorphism. There is a factorization $A \stackrel{i}\rightarrow  B'\stackrel{p} \rightarrow B$ such that $i\in \mathcal{I}\mbox{-cell}$ and $p\in \mathcal{I}^\perp$. Then $i$ is a cofibration and $p$ is a surjective quasi-isomorphism.

We apply the small object argument to $\mathcal{J}$, and obtain another factorization $A \stackrel{j}\rightarrow  A'\stackrel{q} \rightarrow B$ such that $j\in \mathcal{J}\mbox{-cell}$ and $q\in \mathcal{J}^\perp$. Then $j$ lies in $\mathcal{C}of\cap \mathcal{Q}uas$ and $q$ is surjective.

For (MC4), it suffices to show that $(\mathcal{C}of\cap \mathcal{Q}uas)\perp \mathcal{S}urj$. We claim that each $f\colon A\rightarrow B\in (\mathcal{C}of\cap \mathcal{Q}uas)$ is a retract of some element in $\mathcal{J}\mbox{-cell}$. By $\mathcal{J}\mbox{-cell}\subseteq {^\perp\mathcal{S}urj}$, we infer the required orthogonality.

The following argument resembles the one in the proof of Corollary~\ref{cor:retract}.  Take the above second factorization $A \stackrel{j}\rightarrow  A'\stackrel{q} \rightarrow B$ with $j\in \mathcal{J}\mbox{-cell}$ and $q\in \mathcal{S}urj$. It follows that $q$ is also a quasi-isomorphism. Since $f\in \mathcal{C}of$ and $q\in \mathcal{Q}uas\cap \mathcal{S}urj$, the following commutative diagram admits a lifting $h\colon B\rightarrow A'$.
\[
\xymatrix{
A\ar[d]_-{f} \ar[r]^-j & A'\ar[d]^-{q}\\
B\ar@{=}[r] &  B
}\]
Then by the following commutative diagram,
\[
\xymatrix{
A\ar[d]_-f \ar@{=}[r] & A\ar[d]^-j \ar@{=}[r] & A\ar[d]^-{f}\\
B\ar[r]^-h& A' \ar[r]^-q & B
}\]
we infer that $f$  is a retract of $j$, as required. \hfill $\square$

\subsection{Cofibrant dg algebras and homotopies}
We will study the model category ${\rm DGAlg}$. It is clear that any dg algebra is fibrant.

\begin{defn}
A dg algebra $A$ is called \emph{semi-free}, if it is a free graded algebra $A=\mathbb{K}\langle x\; |\; x\in I \rangle$ such that the homogeneous generating set $I$ has a disjoint decomposition $I=\bigsqcup_{p\geq 0} I_p$ and that
$$d_A(x)\in \mathbb{K}\langle y\; |\; y\in \bigsqcup_{q<p} I_q \rangle$$
for any $x\in I_p$. In particular, $d_A(x)=0$ for any $x\in I_0$.
\end{defn}

Recall that $\mathcal{I}=\{\iota_n\colon S(n)\rightarrow D(n+1)\; |\;  n\in \mathbb{Z}\}$.

\begin{prop}
Let $A$ be a dg algebra. Then $A$ is semi-free if and only if $\mathbb{K}\rightarrow A$ lies in $\mathcal{I}\mbox{-{\rm cell}}$. Consequently, $A$ is cofibrant if and only if $A$ is a retract of a semi-free dg algebra.
\end{prop}

\begin{proof}
 For the first statement, we just apply Example~\ref{exm:two-po}(2) repeatedly. For the second one, we apply Lemma~\ref{lem:dgalg-cof} and Remark~\ref{rem:retract}.
\end{proof}

Let $B$ be a dg algebra. Consider the $0$-th disc $D(t)$ and the free product $B\ast D(t)$. For $i=0, 1$,  we have dg algebra homomorphisms $p_i\colon B\ast D(t)\rightarrow B$ such that $p_i|_{B}={\rm Id}_B$ and $p_i(t)=i$. Then we obtain a very good path object for $B$.
$$B\xrightarrow{\rm inc} B\ast D(t) \xrightarrow{\begin{pmatrix}p_0\\ p_1\end{pmatrix}} B\times B$$

The following is inspired by \cite[6.1~Definition]{BG}.

\begin{defn}
Two homomorphisms $f, g\colon A\rightarrow B$ are said to be \emph{elementarily homotopic}, denoted by $f\stackrel{e}\sim g$, if there is a homomorphism $K\colon A\rightarrow B\ast D(t)$ such that $p_0\circ K=f$ and $p_1\circ K=g$.
\end{defn}

We have another path object; see \cite[Subsection~3.1]{Kel99}. Consider the following formal matrix algebra
$$\Gamma=\begin{pmatrix} B & \Sigma^{-1}(B)\\ 0 & B \end{pmatrix},$$
whose multiplication is given by
$$\begin{pmatrix} b_0 & s^{-1}a\\ 0 & b_1 \end{pmatrix} \cdot \begin{pmatrix} x_0 & s^{-1}y\\ 0 & x_1 \end{pmatrix}=\begin{pmatrix} b_0x_0 & s^{-1}((-1)^{|b_0|} b_0y+ax_1)\\ 0 & b_1x_1 \end{pmatrix}.$$
Its differential is given by
$$d_\Gamma \begin{pmatrix} b_0 & s^{-1}a\\ 0 & b_1 \end{pmatrix}=\begin{pmatrix} d_B(b_0) & s^{-1}(-d_B(a)+b_0-b_1) \\ 0 & d_B(b_1) \end{pmatrix}.$$
We might call $\Gamma$ the \emph{cocylinder} of $B$. For $i=0, 1$, we have homomorphisms $\pi_i\colon \Gamma \rightarrow B$ sending $\begin{pmatrix} b_0 & s^{-1}a\\ 0 & b_1 \end{pmatrix}$ to $b_i$. The diagonal map ${\rm diag}\colon B\rightarrow \Gamma$ sends $b$ to $\begin{pmatrix} b & 0\\ 0 & b \end{pmatrix}$. Then we have a good path object for $B$.
$$B\xrightarrow{\rm diag} \Gamma=\begin{pmatrix} B & \Sigma^{-1}(B)\\ 0 & B \end{pmatrix} \xrightarrow{\begin{pmatrix} \pi_0\\ \pi_1\end{pmatrix}} B\times B $$

\begin{defn}
Two homomorphisms $f, g\colon A\rightarrow B$ are said to be \emph{cochain homotopic}, denoted by $f\stackrel{c}\sim g$, if there exists a homomorphism $K'\colon A\rightarrow \Gamma$ such that $\pi_0\circ K'=f$ and $\pi_1\circ K'=g$.
\end{defn}

\begin{lem}
Assume that $f, g\colon A\rightarrow B$ are two homomorphisms. Then $f\stackrel{c}\sim g$ if and only if there exist an $(f, g)$-derivation $\Delta\colon A\rightarrow B$ of degree $-1$ such that $f-g=d_B\circ \Delta+\Delta\circ d_A$.
\end{lem}

Here, a map $\Delta\colon A\rightarrow B$ of degree $-1$ is called an \emph{$(f, g)$-derivation} if
$$\Delta(ab)=\Delta(a)g(b)+(-1)^{|a|}f(a)\Delta(b)$$
for any $a, b\in A$.

\begin{proof}
We just mention that, given such a derivation $\Delta$, we have a homomorphism
$$K'\colon A\longrightarrow \Gamma, \; a\mapsto \begin{pmatrix} f(a) & s^{-1}\Delta(a) \\
                                                    0 & g(a)\end{pmatrix}.$$
                                                    It gives rise to the corresponding cochain homotopy from $f$ to $g$.
\end{proof}

Recall that $\stackrel{r}\sim$ denotes the right homotopy relation.

\begin{prop}
Assume that $f, g\colon A\rightarrow B$ are two homomorphisms. Then the following statements hold.
\begin{enumerate}
\item $f\stackrel{e}\sim g \; \Rightarrow f\stackrel{c}\sim g\;  \Rightarrow\;  f\stackrel{r}\sim g$.
\item Assume that $A$ is cofibrant. Then $\stackrel{e}\sim$, $\stackrel{c}\sim$ and $\stackrel{r}\sim$ coincide.
\end{enumerate}
\end{prop}

\begin{proof}
For (1), it suffices to have the following observation: there is a homomorphism $\phi\colon B\ast D(t)\rightarrow \Gamma$,  which is uniquely determined by
$$\phi|_B={\rm diag}\mbox{ and } \phi(t)=\begin{pmatrix} 0 & 0\\ 0 & 1\end{pmatrix};$$
moreover, we have $\pi_i\circ \phi=p_i$.  The statement (2) follows from Remark~\ref{rem:homo-fixed}.
\end{proof}

\begin{rem}
Let $B=\mathbb{K}$ and $\Gamma=\begin{pmatrix} \mathbb{K} & \Sigma^{-1}(\mathbb{K})\\ 0 & \mathbb{K} \end{pmatrix}$. Consider the homomorphisms $\pi_i\colon \Gamma\rightarrow \mathbb{K}$ for $i=0, 1$.  Then $\pi_0\stackrel{c}\sim \pi_1$ by the very definition. However, it is easy to see that $\pi_0$ and $\pi_1$ are not elementarily homotopic. We point out that $\Gamma$ is not cofibrant.
\end{rem}

\section{The category of small dg categories}

In this section, we study the Dwyer-Kan model structure on the category of small dg categories. We fix a commutative unital ring $\mathbb{K}$, and denote by ${\rm DGCat}$ the category of small dg categories over $\mathbb{K}$. The main references are \cite{Tab05} and \cite[Chapter~1]{Tab15}.

\subsection{Preliminaries}

Let $\mathcal{C}$ be  a small dg category. For any two objects $X$ and $Y$, the differential on the Hom complex $\mathcal{C}(X, Y)$ is usually denoted by $d_\mathcal{C}$. By default, we only consider homogeneous morphisms in dg categories

A morphism $f$ in $\mathcal{C}$ is called \emph{closed} if $d(f)=0$. Denote by $Z^0(\mathcal{C})$ the \emph{ordinary category} of $\mathcal{C}$, which has the same objects as $\mathcal{C}$ and whose morphisms are given by $(Z^0\mathcal{C})(X, Y)=Z^0(\mathcal{C}(X, Y))$. An isomorphism in $Z^0(\mathcal{C})$ is called a \emph{dg-isomorphism} in $\mathcal{C}$.

Denote by $H^0(\mathcal{C})$ the \emph{homotopy category} of $\mathcal{C}$, which has the same objects as $\mathcal{C}$ and whose morphisms are given by $(H^0\mathcal{C})(X, Y)=H^0(\mathcal{C}(X, Y))$. A morphism $f\colon X\rightarrow Y$ in $Z^0(\mathcal{C})$ is called a \emph{homotopy equivalence} if it represents an isomorphism in $H^0(\mathcal{C})$. This is equivalent to the following condition: there exists a closed morphism $g\colon Y\rightarrow X$ of degree zero such that $g\circ f-{\rm Id}_X=d_\mathcal{C}(h_X)$ and $f\circ g-{\rm Id}_Y=d_\mathcal{C}(h_Y)$ for some morphisms $h_X\in \mathcal{C}(X, X)^{-1}$ and $h_Y\in \mathcal{C}(Y, Y)^{-1}$. In such a situation, $g$ is called a \emph{homotopy inverse} of $f$.

\begin{lem}\label{lem:homo-inv}
Assume that $f\colon X\rightarrow Y$ is a homotopy equivalence with a homotopy inverse $g$. Then there exist $h_X\in \mathcal{C}(X, X)^{-1}$, $h_Y\in \mathcal{C}(Y, Y)^{-1}$ and $r\in \mathcal{C}(X, Y)^{-2}$ satisfying the following conditions£º
$$g\circ f-{\rm Id}_X=d_\mathcal{C}(h_X), \; f\circ g-{\rm Id}_Y=d_\mathcal{C}(h_Y), \mbox{ and }h_Y\circ f-f\circ h_X=d_\mathcal{C}(r).$$
\end{lem}

\begin{proof}
By assumption, we may take $h'_X\in \mathcal{C}(X, X)^{-1}$ and $h_Y\in \mathcal{C}(Y, Y)^{-1}$ such that
$$g\circ f-{\rm Id}_X=d_\mathcal{C}(h'_X), \mbox{ and } f\circ g-{\rm Id}_Y=d_\mathcal{C}(h_Y).$$
The following cochain map
$$\mathcal{C}(X, X)\longrightarrow \mathcal{C}(X, Y), \; h\mapsto f\circ h$$
is a homotopy equivalence and thus a quasi-isomorphism. We observe that $h_Y\circ f-f\circ h'_X$ is a cocycle in $\mathcal{C}(X, Y)$ of degree $-1$. By the quasi-isomorphism, we infer that there exists a closed morphism $h''\in \mathcal{C}(X, X)^{-1}$ such that
$$h_Y\circ f-f\circ h'_X=f\circ h''+d_\mathcal{C}(r)$$
for some $r\in \mathcal{C}(X, Y)^{-2}$. We set $h_X=h'_X+h''$ and complete the proof.
\end{proof}

Denote by ${\rm DGMod}\mbox{-}\mathcal{C}$ the dg category of \emph{right} dg $\mathcal{C}$-modules. For each object $X$, we write $\mathcal{C}(-, X)=X^{\wedge}$. We identify $X^\wedge$ with the ``formal" sum $\bigoplus_{T\in {\rm Obj}(\mathcal{C})} \mathcal{C}(T, X)$, whose differential $d_{X^\wedge}$ is given by $d_\mathcal{C}$. For any morphism $a\colon T'\rightarrow T$ in $\mathcal{C}$ and any element $x\in X^\wedge(T)$, we have
$$X^\wedge(a)(x)=(-1)^{|a|\cdot |x|} x\circ a\in X^\wedge(T').$$
The \emph{Yoneda embedding}
$$\mathcal{C}\longrightarrow {\rm DGMod}\mbox{-}\mathcal{C}, \; X\mapsto X^\wedge$$
is a fully faithful dg functor. For a morphism $f\colon X\rightarrow Y$, we describe $f^\wedge \colon X^\wedge \rightarrow Y^\wedge$ as follows: $f^\wedge(x)=f\circ x$ for any $x\in X^\wedge(T)$. The reader should not confuse $f^\wedge$ with $X^\wedge(f)$.

Let $f\colon X\rightarrow Y$ be a closed morphism of degree zero in $\mathcal{C}$. The \emph{mapping cone} of $f^\wedge\colon X^\wedge \rightarrow Y^\wedge$ is a right dg $\mathcal{C}$-module described as follows:
$${\rm Cone}(f^\wedge)=Y^\wedge \oplus \Sigma X^\wedge, $$
whose differential $d_{{\rm Cone}(f^)}$ is given by
\[\begin{pmatrix}
d_{Y^\wedge} & f^\wedge\circ \xi^{-1}\\
0& d_{\Sigma X^\wedge}
\end{pmatrix}.\]
Here, $\Sigma$ denotes the suspension functor, and $\xi\colon X^\wedge \rightarrow \Sigma X^\wedge$ is the canonical isomorphism of degree $-1$. We emphasize that $d_{\Sigma X^\wedge}=-\xi \circ d_{X^\wedge}\circ \xi^{-1}$.

\begin{rem}\label{rem:contrac}
Let $f\colon X\rightarrow Y$ be a homotopy equivalence in $\mathcal{C}$. We use the notation in Lemma~\ref{lem:homo-inv}. Set
\[c=\begin{pmatrix} -{h_Y}^\wedge & r^\wedge \circ \xi^{-1}\\
                     \xi\circ g^\wedge & \xi\circ {h_X}^\wedge\circ \xi^{-1}
                     \end{pmatrix}.\]
                     Then $c$ is a \emph{contracting homotopy} of ${\rm Cone}(f^\wedge)$, that is, $d(c)={\rm Id}_{{\rm Cone}(f^\wedge)}$. The notation $d$ here means the differential in ${\rm DGMod}\mbox{-}\mathcal{C}$. In practice, we have $d(c)=d_{{\rm Cone}(f^\wedge)}\circ c+c\circ d_{{\rm Cone}(f^\wedge)}$.
\end{rem}

\subsection{Limits and colimits}

Let $\mathcal{C}$ be  small dg category. A \emph{dg ideal} $\mathcal{I}$ of $\mathcal{C}$ is determined by the data $(\mathcal{I}(X, Y))_{X, Y\in {\rm Obj}(\mathcal{C})}$, where each $\mathcal{I}(X, Y)\subseteq \mathcal{C}(X, Y)$ is a sub complex with the following property: for any morphisms $f\colon X'\rightarrow X$, $g\in Y\rightarrow Y'$ and $a\in \mathcal{I}(X, Y)$, we have $g\circ a\circ f\in \mathcal{I}(X', Y')$. We denote by $\mathcal{C}/\mathcal{I}$ the \emph{factor dg category}, whose objects are the same as the ones in $\mathcal{C}$ and $(\mathcal{C}/\mathcal{I})(X, Y)=\mathcal{C}(X, Y)/{\mathcal{I}(X, Y)}$ is the quotient complex.

Let $\mathcal{M}$ be a dg $\mathcal{C}$-$\mathcal{C}$-bimodule. We usually identify $\mathcal{M}$ with the ``formal" sum $$\bigoplus_{X, Y\in {\rm Obj}(\mathcal{C})} \mathcal{M}(X, Y),$$
where $\mathcal{M}(X, Y)$ is covariant in $Y$ and contravariant in $X$. In other words, for any morphisms  $f\colon X'\rightarrow X$, $g\in Y\rightarrow Y'$ and any element $x\in \mathcal{M}(X, Y)$, we have $g.x.f\in \mathcal{M}(X', Y')$. Here, we use the dot to denote the $\mathcal{C}$-actions. We observe that $\mathcal{M}(X, Y)$ is a dg $\mathcal{C}(Y, Y)$-$\mathcal{C}(X, X)$-bimodule. For example, we always view $\mathcal{C}$ as a dg $\mathcal{C}$-$\mathcal{C}$-bimodule, whose $\mathcal{C}$-actions are given by the composition of morphisms.

Associated to a dg $\mathcal{C}$-$\mathcal{C}$-bimodule $\mathcal{M}$, we have the \emph{tensor dg category} $T=T_\mathcal{C}(\mathcal{M})$. It has the same object as $\mathcal{C}$, and the Hom complex is
given by the following coproduct of complexes:
\begin{align*}
T&(X, Y)=\mathcal{C}(X, Y)\oplus \mathcal{M}(X, Y)\oplus\\
&\bigoplus_{n\geq 1, U_1, \cdots, U_n\in {\rm Obj}(\mathcal{C})} \mathcal{M}(U_n, Y)\otimes_{\mathcal{C}(U_n, U_n)} \cdots \otimes_{\mathcal{C}(U_2, U_2)} \mathcal{M}(U_1, U_2) \otimes_{\mathcal{C}(U_1, U_1)}\mathcal{M}(X, U_1).
\end{align*}

Let $\alpha\colon \mathcal{M}\rightarrow \mathcal{C}$ be a closed morphism of dg $\mathcal{C}$-$\mathcal{C}$-bimodules of degree one. The \emph{deformed tensor dg category} $T_\mathcal{C}(\mathcal{M}; \alpha)$ is defined as follows: as a graded category, it is the same with $T_\mathcal{C}(\mathcal{M})$; its differential $d'$ is determined by $d'(f)=d_\mathcal{C}(f)$ for any morphism $f$ in $\mathcal{C}$, and $d'(x)=d_\mathcal{M}(x)+\alpha(x)$ for any element $x\in \mathcal{M}$.

\begin{exm}\label{exm:free-add}
Let $f\colon X\rightarrow Y$ be a closed morphism in $\mathcal{C}$. Consider a stalk complex $\mathbb{K}t$ consisting of a free $\mathbb{K}$-module generated by $t$ such that $|t|=|f|-1$. Then $\mathcal{C}(Y, -)\otimes \mathbb{K}t\otimes \mathcal{C}(-, X)$ is naturally a dg $\mathcal{C}$-$\mathcal{C}$-bimodule. We have a morphism of dg $\mathcal{C}$-$\mathcal{C}$-bimodules of degree one:
$$\alpha\colon \mathcal{C}(Y, -)\otimes \mathbb{K}t\otimes \mathcal{C}(-, X)\longrightarrow \mathcal{C}, \quad h\otimes t\otimes k\mapsto (-1)^{|h|} h\circ f\circ k. $$
The deformed dg tensor category
$$T_\mathcal{C}(\mathcal{C}(Y, -)\otimes \mathbb{K}t\otimes \mathcal{C}(-, X);\alpha)$$
 will be denoted by $\mathcal{C}\langle t; f\rangle$. Indeed, it is obtained from $\mathcal{C}$ by \emph{freely} adding a morphism $t\colon X\rightarrow Y$ such that $d(t)=f$. Here, we abuse $t$ with ${\rm Id}_Y\otimes t\otimes {\rm Id}_X$.

We will later consider the following special case $\mathcal{C}\langle c; {\rm Id}_X\rangle$, which is denoted by $\mathcal{C}/\{X\}$; see \cite[Definition~3.1]{Dri}. It is obtained from $\mathcal{C}$ by freely adding a contracting homotopy $c$ for $X$, namely, $c \colon X \rightarrow X$ is of degree $-1$ and satisfies $d(c)={\rm Id}_X$.
\end{exm}

A dg category $\mathcal{C}$ is said to be \emph{discrete} if $\mathcal{C}(X, Y)=0$ for distinct objects $X$ and $Y$, and $\mathcal{C}(X, X)=\mathbb{K}{\rm Id}_X$, a free $\mathbb{K}$-module concentrated in degree zero. We always view $\mathbb{K}$ as a discrete dg category with a single object $\{\ast\}$.

We mention that the empty dg category $\emptyset$ is the initial object in ${\rm DGCat}$. The terminal object is the \emph{zero category} $0$, which consists of a single object with its endomorphism dg algebra being zero.

\begin{prop}
The category ${\rm DGCat}$ has small limits and colimits.
\end{prop}

Let $\Lambda$ be a small index category. We will consider a $\Lambda$-diagram $\mathcal{C}\colon \Lambda\rightarrow {\rm DGCat}$. In other words, for each object $i\in \Lambda_0$, $\mathcal{C}_i$ is a dg category; for any morphism $\alpha\colon i\rightarrow j\in \Lambda_1$, $\mathcal{C}_\alpha\colon \mathcal{C}_i\rightarrow \mathcal{C}_j$ is a dg functor.

\begin{proof}
We will describe the limit and colimit dg categories of any $\Lambda$-diagram $\mathcal{C}\colon \Lambda\rightarrow {\rm DGCat}$.

(1) The limit dg category $\mathcal{L}$ is given as follows:
\begin{align*}
{\rm Obj}(\mathcal{L})&={\rm lim}_{i\in \Lambda_0} \; {\rm Obj}(\mathcal{C}_i)\\
&=\{(X_i)_{i\in \Lambda_0}\in \prod_{i\in \Lambda_0} {\rm Obj}(\mathcal{C}_i)\; |\; \mathcal{C}_\alpha(X_i)=X_j \mbox{ for any } \alpha\colon i\rightarrow j \mbox{  in } \Lambda_1\};
\end{align*}
the Hom complex from $(X_i)_{i\in \Lambda_0}$ to $(Y_i)_{i\in \Lambda_0}$ is given by
\begin{align*}
\mathcal{L}((X_i)_{i\in \Lambda_0},& (Y_i)_{i\in \Lambda_0})={\rm lim}_{i\in \Lambda_0} \; \mathcal{C}_i(X_i, Y_i)\\
&=\{(f_i)_{i\in \Lambda_0}\in \prod_{i\in \Lambda_0}\mathcal{C}_i(X_i, Y_i)\; |\; \mathcal{C}_\alpha(f_i)=f_j  \mbox{ for any } \alpha\colon i\rightarrow j \mbox{  in } \Lambda_1\}.
\end{align*}

(2) The colimit dg category $\mathcal{CL}$ is more subtle.
$${\rm Obj}(\mathcal{CL})={\rm colim}_{i\in \Lambda_0}\; {\rm Obj}(\mathcal{C}_i),$$
whose objects are given by the equivalence class $[X_i]$ for some $X_i\in {\rm Obj}(\mathcal{C}_i)$. We have $[X_i]=[X_j]$ if there exists $\alpha\colon i\rightarrow j$ in $\Lambda_1$ such that $\mathcal{C}_\alpha(X_i)=X_j$. We view $\mathcal{O}={\rm Obj}(\mathcal{CL})$ as a discrete dg category, and have the following dg bimodule.
$$\mathcal{M}([X], [Y])=\bigoplus_{i\in \Lambda_0}\bigoplus_{X', Y'\in {\rm Obj}(\mathcal{C}_i), X'\in [X], Y'\in [Y]}\mathcal{C}_i(X', Y').$$
We form the dg tensor category $T_\mathcal{O}(\mathcal{M})$. Finally, the colimit dg category $\mathcal{CL}$ is given by $T_\mathcal{O}(\mathcal{M})/\mathcal{I}$, where the dg ideal $\mathcal{I}$ is generated by the following morphisms.
\begin{enumerate}
\item $g\circ f-g\otimes f$, for any two composable morphisms $g, f$ in $\mathcal{C}_i$;
\item ${\rm Id}_X-{\rm Id}_{[X]}$, for any object $X$ in $\mathcal{C}_i$;
\item $h-\mathcal{C}_\alpha(h)$, for any morphism $h\in \mathcal{C}_i$ and any $\alpha\colon i\rightarrow j$ in $\Lambda_1$.
\end{enumerate}
Here, in (1), $g\circ f$ denotes the composition in $\mathcal{C}_i$ and $g\otimes f$ is viewed as an element in $T_\mathcal{O}(\mathcal{M})$ of tensor-length two.
\end{proof}

\begin{rem}\label{rem:dgcat-bicom}
(1) Assume that $\Lambda$ is filtered. Then the filtered colimit dg category $\mathcal{CL}$  is easier to describe. We have
$${\rm Obj}(\mathcal{CL})={\rm colim}_{i\in \Lambda_0}\; {\rm Obj}(\mathcal{C}_i)=\bigsqcup_{i\in \Lambda_0} \; {\rm Obj}(\mathcal{C}_i)/{\sim},$$
where the equivalence relation $\sim$ is given as follows: for $X_i\in {\rm Obj}(\mathcal{C}_i)$ and $X_j\in {\rm Obj}(\mathcal{C}_j)$, $X_i\sim X_j$ if and only if there exist $\alpha\colon i\rightarrow k$ and $\beta\colon j\rightarrow k$ in $\Lambda_1$ such that $\mathcal{C}_\alpha(X_i)=\mathcal{C}_\beta(X_j)$. The Hom complex
$$\mathcal{CL}([X], [Y])={\rm colim}_{\{(i, X', Y')\; |\; i\in \Lambda_0, X', Y'\in {\rm Obj}(\mathcal{C}_i) \mbox{ with } [X']=[X], [Y']=[Y]\}}\; \mathcal{C}_i(X', Y'),$$
where the right colimit is taken over a filtered category consisting of triples $(i, X', Y')$, whose morphisms from $(i, X', Y')$ to $(j, X'', Y'')$ are given  by $\alpha\colon i\rightarrow j$ in $\Lambda_1$ satisfying $\mathcal{C}_\alpha(X')=X''$ and $\mathcal{C}_\alpha(Y')=Y''$.

(2) Let $\mathcal{C}$ and $\mathcal{D}$ be two dg categories. Their coproduct $\mathcal{C}\coprod \mathcal{D}$ is described as follows: ${\rm Obj}(\mathcal{C}\coprod \mathcal{D})$ is the disjoint union of ${\rm Obj}(\mathcal{C})$ and ${\rm Obj}(\mathcal{D})$; $(\mathcal{C}\coprod \mathcal{D})(C, D)=0=(\mathcal{C}\coprod \mathcal{D})(D, C)$ for any object $C$ in $\mathcal{C}$ and any object $D$ in $\mathcal{D}$; moreover, both $\mathcal{C}$ and $\mathcal{D}$ are full dg subcategories of $\mathcal{C}\coprod \mathcal{D}$. For example, $\mathcal{C}\coprod \mathbb{K}$ is obtained from $\mathcal{C}$ by adding a new discrete object $\ast$. We mention that by definition the following diagram is a pushout.
\[
\xymatrix{
\emptyset \ar[d]\ar[rr] && \mathcal{C}\ar[d]\\
\mathbb{K} \ar[rr] && \mathcal{C}\coprod \mathbb{K}
}\]
\end{rem}

\subsection{Semi-free dg categories}

We will study semi-free dg categories. Let us begin with spheres and discs.

\begin{defn}
For each integer $n$, we define the $n$-th \emph{sphere} $\mathcal{S}(n)$, which is a dg category given as follows: ${\rm Obj}(\mathcal{S}(n))=\{1, 2\}$, $\mathcal{S}(n)(i, i)=\mathbb{K}{\rm Id}_i$ for $i=1, 2$, and $\mathcal{S}(n)(2, 1)=0$ and $\mathcal{S}(n)(1, 2)=\mathbb{K}u$, a free $\mathbb{K}$-module with a generator $u$ of degree $-n$. The differential is determined by $d_{\mathcal{S}(n)}(u)=0$. To emphasize the generator $u$, we often denote $\mathcal{S}(n)$ by $\mathcal{S}(u)$.
\end{defn}

For any dg category $\mathcal{C}$, we have the following bijection:
$${\rm DGCat}(\mathcal{S}(n), \mathcal{C})\simeq \bigsqcup_{X, Y\in {\rm Obj}(\mathcal{C})} Z^{-n}\mathcal{C}(X, Y), \quad F\mapsto F(u).$$

\begin{defn}
The $n$-th \emph{disc} $\mathcal{D}(n)$ is a dg category defined as follows: ${\rm Obj}(\mathcal{D}(n))=\{1, 2\}$, $\mathcal{D}(n)(i, i)=\mathbb{K}{\rm Id}_i$ for $i=1, 2$, and $\mathcal{D}(n)(2, 1)=0$ and $\mathcal{D}(n)(1, 2)=\mathbb{K}t\oplus \mathbb{K}dt$, a free $\mathbb{K}$-modules with generators $t$ and $dt$ such that $|t|=-n$ and $|dt|=1-n$. The differential is determined by $d_{\mathcal{D}(n)}(t)=dt$ and $d_{\mathcal{D}(n)}(dt)=0$. To emphasize the generator $t$, we sometimes denote $\mathcal{D}(n)$ by $\mathcal{D}(t)$.
\end{defn}

Similar to the above bijection, we have the following one:
\begin{align}\label{equ:disc}
{\rm DGCat}(\mathcal{D}(n), \mathcal{C})\simeq \bigsqcup_{X, Y\in {\rm Obj}(\mathcal{C})} \mathcal{C}^{-n}(X, Y), \quad F\mapsto F(t).
\end{align}

We have a dg functor $\iota_n\colon \mathcal{S}(n) \rightarrow \mathcal{D}(n+1)$, which sends $i$ to $i$, and  $u$ to $dt$. Similarly, we have another dg functor $\pi_n\colon \mathcal{D}(n)\rightarrow \mathcal{S}(n)$ sending $t$ to $u$, and $dt$ to $0$. Then the following commutative diagram
\[\xymatrix{
\mathcal{S}(n)\ar[d]_-{\iota_n} \ar[rr] && \mathbb{K}\coprod\mathbb{K} \ar[d]\\
\mathcal{D}(n+1) \ar[rr]^-{\pi_{n+1}} && \mathcal{S}(n+1)
}\]
is a pushout and also a pullback. Here, the coproduct $\mathbb{K}\coprod\mathbb{K}$ is a discrete dg category with two objects $\{\ast_1, \ast_2\}$. The upper horizontal arrow sends $i$ to $\ast_i$, and the right vertical arrow sends $\ast_i$ to $i$.

\begin{exm}\label{exm:free-add-po}
Let $\mathcal{C}$ be a dg category and  $f$ be a closed morphism of degree $-n$. Recall from Example~\ref{exm:free-add} the dg category $\mathcal{C}\langle t; f \rangle$ obtained from $\mathcal{C}$ by freely adding a morphism $t$ satisfying $d(t)=f$. The following diagram is a pushout.
\[
\xymatrix{
\mathcal{S}(n)\ar[d]_-{\iota_n} \ar[rr]^-{u\mapsto f} && \mathcal{C} \ar[d]^-{\rm inc}\\
\mathcal{D}(n+1)\ar[rr]^-{t\mapsto t, \; dt\mapsto f} && \mathcal{C}\langle t; f \rangle.
}\]
We observe that $\mathbb{K}\coprod\mathbb{K}\langle u; 0\rangle\simeq \mathcal{S}(n+1)$, where $0\colon \ast_1\rightarrow \ast_2$ is the zero morphism  and $|u|=-n-1$.
\end{exm}

By a \emph{graded quiver}, we mean a quiver $Q=(Q_0, Q_1; s,t)$ together with a degree map $|-|\colon Q_1\rightarrow \mathbb{Z}$. The \emph{graded path category} $\mathbb{K}Q$ is defined to be the $\mathbb{K}$-linearization of the path category $\mathcal{P}(Q)$ whose grading is induced by the degree map.

\begin{defn}
A dg category $\mathcal{C}$ is called \emph{semi-free}, if there is a graded quiver $Q$ such that $\mathcal{C}=\mathbb{K}Q$ as a graded category, and that there is a filtration of $Q_1$:
$$Q_1^{(0)}\subseteq Q_1^{(1)} \subseteq Q_1^{(2)}\subseteq \cdots$$
with the property that $Q_1=\bigcup_{i\geq 0}Q_1^{(i)}$ and $d(Q_1^{(i)})\subseteq \mathbb{K}Q^{(i-1)}$ for any $i\geq 0$.  Here, $Q^{(i-1)}$ denotes the sub quiver $(Q_0, Q_1^{(i-1)}; s, t)$ and $\mathbb{K}Q^{(-1)}=0$.
\end{defn}

For example, both $\mathcal{S}(n)$ and $\mathcal{D}(n)$ are semi-free.  We set $\mathcal{I}=\{\emptyset\rightarrow \mathbb{K}, \; \iota_n\colon \mathcal{S}(n)\rightarrow \mathcal{D}(n+1)\; |\;  n\in \mathbb{Z}\}$.

\begin{prop}
A dg category $\mathcal{C}$ is semi-free if and only if $\emptyset\rightarrow \mathcal{C}$ belongs to $\mathcal{I}\mbox{-}{\rm cell}$.
\end{prop}

\begin{proof}
Use Remark~\ref{rem:dgcat-bicom}(2) and Example~\ref{exm:free-add-po}.
\end{proof}

The following construction is to add a ``formal" cone of a closed morphism of degree zero.

\begin{exm}\label{exm:add-cone}
Let $\mathcal{C}$ be a dg category and $f\colon X\rightarrow Y$ be a closed morphism of degree zero. Applying Example~\ref{exm:free-add-po} to $\mathcal{C}\coprod \mathbb{K}$ repeatedly, we form a new dg category $\mathcal{C}' $  by freely adding four morphisms in following diagram.
 \[
 \xymatrix{
 X  \ar@<-0.7ex>[dr]_-{j}\ar[rr]^-f && Y \ar@<-.7ex>[ld]_-{i}\\
 & \ast \ar@<-.7ex>[lu]_-{p} \ar@<-.7ex>[ru]_-{q}
 }\]
 These gradings are given by  $|i|=0=|q|$, $|p|=1$ and $|j|=-1$; the differential is determined by $d(i)=0=d(p)$, and $d(q)=-f\circ p$ and $d(j)=i\circ f$. We should be cautious about the minus sign in $d(q)$.

 The quotient dg category of $\mathcal{C}'$ modulo the dg ideal generated by
 $$ q\circ i-{\rm Id}_Y, p\circ j-{\rm Id}_X, \mbox{and } j\circ p+i\circ q-{\rm Id}_\ast$$
  is denoted by $\mathcal{C}\vee {\rm Cone}(f)$. It is a nice exercise to show that both $p\circ i$ and $q\circ j$ lie in the ideal above.

   This notation $\mathcal{C}\vee {\rm Cone}(f)$ is justified by the fact that in this dg category, the new object $\ast$ is the cone of $f$, or equivalently, in the dg category of dg modules $\ast^\wedge$ is dg-isomorphic to ${\rm Cone}(f^\wedge)$.
\end{exm}

\begin{lem}\label{lem:add-cone}
The natural functor $\mathcal{C}\rightarrow \mathcal{C}\vee {\rm Cone}(f)$ is fully faithful.
\end{lem}

\begin{proof}
It is not hard to show that the functor is full by induction on the times of a morphism passing through $\ast$.  For the faithfulness, we observe that the Yoneda embedding of $\mathcal{C}$ factors though this functor.
\end{proof}

\begin{rem}\label{rem:add-cone}
Write $\mathcal{C}'=\mathcal{C}\vee {\rm Cone}(f)$. For any object $Z$ in $\mathcal{C}'$, we observe that $\mathcal{C}'(Z, \ast)$ is isomorphic to the mapping cone of $\mathcal{C}'(Z, f)\colon \mathcal{C}'(Z, X)\rightarrow \mathcal{C}'(Z, Y)$. In more details, it sends $x\in \mathcal{C}'(Z, \ast)$ to $\binom{q\circ x}{-s(p\circ x)}\in {\rm Cone}(\mathcal{C}'(Z, f))$. Here, any element in $\Sigma \mathcal{C}'(Z, X)$ is written as $sz$ for $z\in \mathcal{C}'(Z, X)$.

There are two consequences: (1) if $f$ is a homotopy equivalence, then $\ast$ is a \emph{contractible object}, that is, there is a contacting homotopy $c\in \mathcal{C}'(\ast, \ast)^{-1}$ with $d(c)={\rm Id}_\ast$; (2) if $\mathcal{C}(X, f)$ is an isomorphism, then the cochain complex $\mathcal{C}'(X, \ast)$ is contractible.
\end{rem}

\subsection{The dg category $\mathcal{K}$} The following dg category will play a special role; see \cite[Subsection~3.7]{Dri}.

\begin{defn}
We define $\mathcal{K}$ be a dg category given by the following graded quiver
\[\xymatrix{  1 \ar@(dl,ul)[]|{h_1} \ar@<+1.3ex>@/^.7pc/[rr]|{r} \ar@<+.6ex>[rr]|f && 2\ar@<+.9ex>[ll]|g \ar@(dr, ur)[]|{h_2} }\]
such that $|f|=0=|g|$, $|h_1|=-1=|h_2|$ and $|r|=-2$. The differential is determined by $d_\mathcal{K}(f)=0=d_\mathcal{K}(g)$, $d_\mathcal{K}(h_1)=g\circ f-{\rm Id}_1$, $d_\mathcal{K}(h_2)=f\circ g-{\rm Id}_2$ and $d_\mathcal{K}(r)=h_2\circ f-f\circ h_1$.
\end{defn}

\begin{rem}
We observe that
$$g\circ h_2-h_1\circ g=d_\mathcal{K}(g\circ r\circ g+g\circ h_2^2+h_1^2\circ g-h_1\circ g\circ h_2).$$
 This identity also follows from $d(c^2)=c^3$, where $c$ is the corresponding contracting homotopy for ${\rm Cone}(f^\wedge)$; compare Remark~\ref{rem:contrac}.
\end{rem}

In view of Remark~\ref{rem:contrac}, the following result is natural.

\begin{lem}
Let $\mathcal{C}$ be a small dg category. Then we have a bijection
$${\rm DGCat}(\mathcal{K}, \mathcal{C})\simeq \left\{(u, c)\;{\large |} \; \begin{aligned}
& u \mbox{ is a closed morphism of degree zero in } \mathcal{C}, \\
 & c \mbox{ is a contacting homotopy for } {\rm Cone}(u^\wedge).\end{aligned} \right \}$$
\end{lem}

\begin{proof}
The bijection sends a dg functor $F$ to $(F(f), c)$, where $c$ is given by the following matrix.
\[\begin{pmatrix} -{F(h_2)}^\wedge & F(r)^\wedge \circ \xi^{-1}\\
                     \xi\circ F(g)^\wedge & \xi\circ {F(h_1)}^\wedge\circ \xi^{-1}
                     \end{pmatrix}\]
Here, we recall that ${\rm Cone}(u^\wedge)=F(2)^\wedge \oplus \Sigma F(1)^\wedge$ and that $\xi\colon F(1)^\wedge \rightarrow \Sigma F(1)^\wedge$ is the canonical isomorphism of degree $-1$.
\end{proof}

We have a variant of the above bijection.

\begin{lem}\label{lem:fromK}
Let $\mathcal{C}$ be a dg category and $v$ be a closed morphism of degree zero. Write $\mathcal{C}'=\mathcal{C}\vee {\rm Cone}(v)$. Then we have a bijection
$$\{F\in {\rm DGCat}(\mathcal{K}, \mathcal{C})\; |\; F(f)=v\}\simeq \{c\in \mathcal{C}'(\ast, \ast)^{-1}\; |\; d_{\mathcal{C}'}(c)={\rm Id}_\ast\}.$$
\end{lem}

\begin{proof}
We identify the added object $\ast$ in $\mathcal{C}'$ with ${\rm Cone}(v^\wedge)$.
\end{proof}

Recall the $0$-th sphere $\mathcal{S}(0)$ with a generator $u$. We refer to Example~\ref{exm:add-cone} for the construction of adding a cone.

\begin{prop}\label{prop:K-embed}
There is an isomorphism of dg categories:
$$\mathcal{K}\vee {\rm Cone}(f) \simeq (\mathcal{S}(0)\vee {\rm Cone}(u))/{{\rm Cone}(u)}.$$
Consequently, we might view $\mathcal{K}$ as a full dg subcategory of $(\mathcal{S}(0)\vee {\rm Cone}(u))/{{\rm Cone}(u)}$.
\end{prop}

\begin{proof}
The dg category $\mathcal{K}\vee {\rm Cone}(f)$ is given by the following quiver
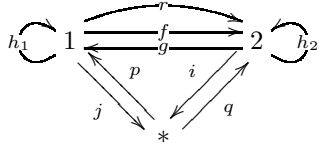
\[\xymatrix{  1  \ar@<-0.7ex>[dr]_-{j}  \ar@(dl,ul)[]|{h_1} \ar@<+1.3ex>@/^.7pc/[rr]|{r} \ar@<+.9ex>[rr]|f && 2\ar@<+.5ex>[ll]|g   \ar@(dr, ur)[]|{h_2}  \ar@<-.7ex>[ld]_-{i} \\
 & \ast \ar@<-.55ex>[lu]_-{p} \ar@<-.7ex>[ru]_-{q}
}\]
with relations $\{q\circ i-{\rm Id}_2, p\circ j-{\rm Id}_1, j\circ p+i\circ q-{\rm Id}_\ast\}$. On the other hand, the dg category $(\mathcal{S}(0)\vee {\rm Cone}(u))/{{\rm Cone}(u)}$ is given by the following quiver
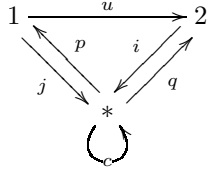
\[\xymatrix{  1  \ar@<-0.7ex>[dr]_-{j}  \ar[rr]^-u && 2   \ar@<-.7ex>[ld]_-{i} \\
 &\ar@(dl,dr)|c  \ast \ar@<-.7ex>[lu]_-{p} \ar@<-.7ex>[ru]_-{q}
}\]
with the same relations. The required dg functor
$$F\colon \mathcal{K}\vee {\rm Cone}(f)\longrightarrow (\mathcal{S}(0)\vee {\rm Cone}(u))/{{\rm Cone}(u)}$$ acts on objects by the identity. On morphisms,  it acts identically on $\{i, p, q, j\}$; moreover,  we have $F(f)=u$, $F(g)=p\circ c \circ i$, $F(h_1)=p\circ c\circ j$, $F(h_2)=-q\circ c\circ i$ and $F(r)=q\circ c\circ j$. We mention that
$$F^{-1}(c)=-i\circ h_2\circ q+i\circ r\circ p+j\circ g\circ q+j\circ h_1\circ p.$$
We leave the details to the reader. The last statement follows from Lemma~\ref{lem:add-cone}.
\end{proof}

\begin{cor}\label{cor:Kquasi}
The inclusion $\mathbb{K}({\rm Id}_1)\rightarrow \mathcal{K}(1, 1)$ is a quasi-isomorphism.
\end{cor}

\begin{proof}
Write $\mathcal{S}=\mathcal{S}(0)$, $\mathcal{S}'=\mathcal{S}\vee {\rm Cone}(u)$  and $\mathcal{S}''=\mathcal{S}'/{{\rm Cone}(u)}$. By the result above, it suffices to show that $\mathbb{K}({\rm Id}_1) \rightarrow \mathcal{S}''(1, 1)$ is a quasi-isomorphism. Since $\mathcal{S}''$ is obtained from $\mathcal{S}'$ by freely adding $c$, we have
$$\mathcal{S}''(1, 1)=\mathcal{S}'(1, 1)\oplus\bigoplus_{n\geq 1} L_n$$
as graded modules, where
$$ L_n= \mathcal{S}'(\ast, 1)\otimes\mathbb{K}c\otimes \mathcal{S}'(\ast, \ast) \otimes \cdots  \otimes \mathbb{K}c\otimes \mathcal{S}'(\ast, \ast)\otimes \mathbb{K}c\otimes \mathcal{S}'(1, \ast)$$
with $c$ appearing $n$ times. Since $d(c)={\rm Id}_\ast$,  as a cochain complex we have the following filtration for $\mathcal{S}'(1, 1)$,
$$\mathcal{S}'(1, 1)\subseteq \mathcal{S}'(1, 1)\oplus L_1\subseteq \mathcal{S}'(1, 1)\oplus L_1\oplus L_2\subseteq \cdots$$
whose factors are $L_n$. Since $\mathcal{S}(1, u)\colon \mathcal{S}(1, 1)\rightarrow \mathcal{S}(1, 2)$ is an isomorphism, $\mathcal{S}'(1, \ast)$ is contractible by Remark~\ref{rem:add-cone}. It follows that each $L_n$ is contractible. Using the filtration, we infer that the inclusion $\mathbb{K}({\rm Id}_1)=\mathcal{S}'(1,1)\rightarrow \mathcal{S}''(1, 1)$ is a quasi-isomorphism.
\end{proof}

\subsection{The Dwyer-Kan model structure}
We will study the Dwyer-Kan model structure on ${\rm DGCat}$, which is analogous to the usual model structure on the category of small simplicial categories.

\begin{defn}
A dg functor $F\colon \mathcal{C}\rightarrow \mathcal{D}$ is  \emph{quasi-fully faithful}, if for any objects $X, Y\in \mathcal{C}$, the induced cochain map
$$\mathcal{C}(X, Y)\longrightarrow \mathcal{D}(F(X), F(Y)), \; f\mapsto F(f)$$
is a quasi-isomorphism. If in addition $H^0(F)\colon H^0(\mathcal{C})\rightarrow H^0(\mathcal{D})$ is dense, then $F$ is called a \emph{quasi-equivalence}.
\end{defn}

We denote by $\mathcal{Q}equ$ the class formed by all quasi-equivalences.

\begin{defn}
A dg functor $F\colon \mathcal{C}\rightarrow \mathcal{D}$ is called a \emph{full isofibration}, if it satisfies the following two conditions:
\begin{enumerate}
\item it is full, that is, for for any objects $X, Y\in \mathcal{C}$, the induced cochain map
$\mathcal{C}(X, Y)\rightarrow \mathcal{D}(F(X), F(Y))$ is surjective;
\item for any homotopy equivalence $f\colon F(X)\rightarrow Y'$ in $\mathcal{D}$ with $X\in {\rm Obj}(\mathcal{C})$, there exists a homotopy equivalence $f'\colon X\rightarrow X'$ in $\mathcal{C}$ such that $F(f')=f$; in particular, $F(X')=Y'$.
\end{enumerate}
\end{defn}

We denote by $\mathcal{F}isof$ the class formed by all full isofibrations, and set $$\mathcal{C}of={^\perp(\mathcal{Q}equ \cap \mathcal{F}isof)}.$$

The following fundamental result is due to \cite{Tab05}, which is somehow expected in \cite[B.6~Lemma and Subsection~3.7]{Dri}.

\begin{thm}\label{thm:Tabuada}
The triple $(\mathcal{C}of, \mathcal{Q}equ, \mathcal{F}isof)$ is a model structure on {\rm DGCat}.
\end{thm}

The above model structure is called the \emph{Dwyer-Kan model structure} on {\rm DGCat}; it is a nontrivial mixture of the natural model structure on {\rm Cat} and the model structure on {\rm DGAlg}. We mention other model structures on ${\rm DGCat}$, namely, the \emph{pretriangulated model structure} and the \emph{Morita model structure}; see \cite[Chapter~I]{Tab15}.

We will make some preparation for the proof. The following observation is easy.

\begin{lem} \label{lem:Fisof-Qequ}
Let $F\colon \mathcal{C}\rightarrow \mathcal{D}$ be a dg functor.
Then $F$ belongs to $\mathcal{Q}equ \cap \mathcal{F}isof$ if and only if it is surjective on objects and for any objects $X, Y\in \mathcal{C}$, the induced cochain map $\mathcal{C}(X, Y)\rightarrow \mathcal{D}(F(X), F(Y))$ is a surjective quasi-isomorphism. $\square$
\end{lem}

Recall that $\mathcal{I}=\{\emptyset\rightarrow \mathbb{K}, \; \iota_n\colon \mathcal{S}(n)\rightarrow \mathcal{D}(n+1)\; |\;  n\in \mathbb{Z}\}$.

\begin{lem}\label{lem:Fisof-Qequ-ortho}
We have $\mathcal{Q}equ \cap \mathcal{F}isof=\mathcal{I}^\perp$.
\end{lem}

\begin{proof}
We observe that a dg functor $F$ belongs to $\{\emptyset \rightarrow \mathbb{K}\}^\perp$ if and only if it is surjective on objects; $F$ belongs to $\{\iota_n\colon \mathcal{S}(n)\rightarrow \mathcal{D}(n+1)\; |\;  n\in \mathbb{Z}\}^\perp$ if and only if for any objects $X, Y\in \mathcal{C}$, the induced cochain map $\mathcal{C}(X, Y)\rightarrow \mathcal{D}(F(X), F(Y))$ is a surjective quasi-isomorphism. Here, we apply Proposition~\ref{prop:surj-quas} to $\mathcal{C}(X, Y)\rightarrow \mathcal{D}(F(X), F(Y))$.  Combining the observation with Lemma~\ref{lem:Fisof-Qequ}, we deduce the required equality.
\end{proof}

We will consider the unique dg functor $\sigma \colon \mathbb{K}\rightarrow \mathcal{K}$ sending $\ast$ to $1$. For each $n\in \mathbb{Z}$, we have an inclusion $\rho_n\colon \mathbb{K}\coprod \mathbb{K}\rightarrow \mathcal{D}(n)$ which sends $\ast_i$ to $i$. Here, we identify $\mathbb{K}\coprod \mathbb{K}$ with the discrete dg category consisting of two objects $\ast_1$ and $\ast_2$. Set $\mathcal{J}=\{\rho_n, \sigma\; |\; n\in \mathbb{Z}\}$.

The following two propositions are nontrivial; see \cite[Section~1.2]{Tab15}.

\begin{prop}\label{prop:J-perp}
We have $\mathcal{F}isof=\mathcal{J}^\perp$.
\end{prop}

\begin{proof}
In view of (\ref{equ:disc}), we infer that a dg functor $F$ is full if and only if it belongs to $\{\rho_n\; |\; n\in \mathbb{Z}\}^\perp$.

We assume that $F\colon \mathcal{C}\rightarrow \mathcal{D}$ lies in $\mathcal{J}^\perp$. Then it is full. For any object $X$ in $\mathcal{C}$ and any homotopy equivalence $v\colon F(X)\rightarrow Y$ in $\mathcal{D}$, applying Lemma~\ref{lem:homo-inv} we have a dg functor $G\colon \mathcal{K}\rightarrow \mathcal{D}$ such that $G(f)=v$. Any lifting of the following diagram
\[
\xymatrix{
\mathbb{K}\ar[d]_-\sigma \ar[rr]^-{\ast\mapsto X} && \mathcal{C}\ar[d]^-F \\
\mathcal{K}\ar[rr]^-G && \mathcal{D}
}\]
yields a  homotopy equivalence $v'\colon X\rightarrow X'$ in $\mathcal{C}$ such that $F(v')=v$. We infer that  $F$ is a full isofibration.

Conversely, assume that $F\colon \mathcal{C}\rightarrow \mathcal{D}$ be a full isofibration and that we are given a commutative diagram.
\[
\xymatrix{
\mathbb{K}\ar[d]_-\sigma \ar[rr]^-{\ast\mapsto X} && \mathcal{C}\ar[d]^-F \\
\mathcal{K}\ar[rr]^-G && \mathcal{D}
}\]
The homotopy equivalence $G(f)\colon F(X)\rightarrow Y$ in $\mathcal{D}$ lifts to a homotopy equivalence $v'\colon X\rightarrow X'$ in $\mathcal{C}$. Then $F$ extends naturally to a dg functor
$$F'\colon \mathcal{C}'=\mathcal{C}\vee {\rm Cone}(v')\longrightarrow \mathcal{D}'=\mathcal{D}\vee {\rm Cone}(G(f)).$$
Since $F$ is full, so is $F'$. Therefore, $F'$ induces a surjective quasi-isomorphism
$$F'_{\ast, \ast}\colon \mathcal{C}'(\ast, \ast)\longrightarrow \mathcal{D}'(\ast, \ast),$$
as both $\mathcal{C}'(\ast, \ast)$ and $\mathcal{D}'(\ast, \ast)$ are acyclic; see Remark~\ref{rem:add-cone}. By the bijection in Lemma~\ref{lem:fromK}, the dg functor $G$ corresponds to $c\in \mathcal{D}'(\ast, \ast)^{-1}$ with $d(c)={\rm Id}_\ast$. Applying Proposition~\ref{prop:surj-quas} to $F'_{\ast, \ast}$, we infer that there exists $c'\in \mathcal{C}'(\ast, \ast)^{-1}$ with $d(c')={\rm Id}_\ast$ and $F'(c')=c$. Applying Lemma~\ref{lem:fromK}, we obtain a dg functor $H\colon \mathcal{K}\rightarrow \mathcal{C}$ such that $H(f)=v'$ and that $H$ corresponds to $c'$. It follows from $F'(c')=c$ that $FH=G$. In other words, $H$ is a required lifting of the diagram above.
\end{proof}

\begin{prop}\label{prop:J-cell}
We have $\mathcal{J}\mbox{-{\rm cell}}\subseteq \mathcal{I}\mbox{-{\rm cell}}\cap \mathcal{Q}equ$.
\end{prop}

\begin{proof}
In view of Example~\ref{exm:free-add-po}, we have $\mathcal{J}\subseteq \mathcal{I}\mbox{-cell}$. Since $\mathcal{I}\mbox{-cell}$ is closed under infinite coproducts, pushouts and transfinite compositions.  It follows that $\mathcal{J}\mbox{-cell}\subseteq \mathcal{I}\mbox{-cell}$.

Let $\mathcal{C}$ be a dg category. Consider the following pushout.
\[\xymatrix{
\mathbb{K}\coprod \mathbb{K} \ar[d]_-{\rho_n}\ar[rr]^-{\ast_1\mapsto X,\;  \ast_2\mapsto Y} && \mathcal{C}\ar@{.>}[d]\\
\mathcal{D}(n) \ar@{.>}[rr] && \mathcal{C}_1
}\]
Here, $\mathcal{C}_1=T_\mathcal{C}(\mathcal{M})$ with a dg $\mathcal{C}$-$\mathcal{C}$-bimodule
$$\mathcal{M}=\mathcal{C}(Y, -)\otimes(\mathbb{K}t\oplus \mathbb{K}dt)\otimes\mathcal{C}(-, X),$$
where $\mathbb{K}t\oplus \mathbb{K}dt$ is a contractible cochain complex with $|t|=-n$ and $|dt|=1-n$. It follows that $\mathcal{M}$ is contractible. Then $\mathcal{C}\rightarrow \mathcal{C}_1$ is a quasi-equivalence.

We consider another pushout.
\[\xymatrix{
\mathbb{K} \ar[d]_-{\sigma}\ar[rr]^-{\ast\mapsto X} && \mathcal{C}\ar@{.>}[d]\\
\mathcal{K} \ar@{.>}[rr] && \mathcal{C}_2
}\]
Then $\mathcal{C}_2$ is obtained from $\mathcal{C}$ by adding a new object $2$ and freely adding five morphisms, which are indicated below.
\[\xymatrix{ \begin{xy}*+{X}*\frm{--}\end{xy} \ar@(dl,ul)[]|{h_X} \ar@<+1.3ex>@/^.7pc/[rr]|{r} \ar@<+.6ex>[rr]|f && 2\ar@<+.9ex>[ll]|g \ar@(dr, ur)[]|{h_2} }\]
Here, we use $\begin{xy}*+{X}*\frm{--}\end{xy}$ to represent the given dg category $\mathcal{C}$ with the chosen object $X$. Similar to Proposition~\ref{prop:K-embed}, we infer that $\mathcal{C}_2$ is identified with a full dg subcategory of the dg category $\mathcal{C}_3$, which is  determined by the following diagram
\[\xymatrix{  \begin{xy}*+{X}*\frm{--}\end{xy}  \ar@<-0.7ex>[dr]_-{j}  \ar[rr]^-f && 2   \ar@<-.7ex>[ld]_-{i} \\
 &\ar@(dl,dr)|c  \ast \ar@<-.7ex>[lu]_-{p} \ar@<-.7ex>[ru]_-{q}
}\]
with the relations $\{q\circ i-{\rm Id}_2, p\circ j-{\rm Id}_X, j\circ p+i\circ q-{\rm Id}_\ast\}$. The same argument in the proof of Corollary~\ref{cor:Kquasi} yields that $\mathcal{C}\rightarrow \mathcal{C}_3$ is a quasi-fully faithful. It follows that $\mathcal{C}\rightarrow \mathcal{C}_2$ is a quasi-equivalence.

We observe that $\mathcal{Q}equ$ is closed under transfinite compositions, which are possibly indexed by large ordinals. From the above two quasi-equivalences arising in pushouts, it is not hard to infer  that $\mathcal{J}\mbox{-cell}\subseteq \mathcal{Q}equ$.
\end{proof}

\vskip 5pt

\noindent \emph{Proof of Theorem~\ref{thm:Tabuada}}. \; The verification of (MC1)-(MC3) is easy.

By Lemma~\ref{lem:Fisof-Qequ-ortho}, we have $\mathcal{C}of={^\perp(\mathcal{I}^\perp)}\supseteq \mathcal{I}\mbox{-cell}$. For (MC5), we apply Theorem~\ref{thm:Quillen} to $\mathcal{I}$ and any dg functor $F\colon \mathcal{C}\rightarrow \mathcal{D}$. We obtain a factorization $F=pi$ with $p\in \mathcal{I}^\perp=\mathcal{Q}equ \cap \mathcal{F}isof$ and $i\in \mathcal{I}\mbox{-cell}\subseteq \mathcal{C}of$.

We apply Theorem~\ref{thm:Quillen} to $\mathcal{J}$ and any dg functor $F\colon \mathcal{C}\rightarrow \mathcal{D}$. We obtain a factorization $F=qj$ with $q\in \mathcal{J}^\perp=\mathcal{F}isof$ and $j\in \mathcal{J}\mbox{-cell}$. By Proposition~\ref{prop:J-cell}, we have $j\in \mathcal{I}\mbox{-cell}\cap \mathcal{Q}equ\subseteq \mathcal{C}of \cap \mathcal{Q}equ$.

For (MC4), it remains to prove $(\mathcal{C}of\cap\mathcal{Q}equ)\perp \mathcal{F}isof$. By Proposition~\ref{prop:J-perp} we have $\mathcal{J}\subseteq {^\perp \mathcal{F}isof}$ and thus $\mathcal{J}\mbox{-cell}\subseteq {^\perp \mathcal{F}isof}$. Therefore, it suffices to prove that any dg functor $F\colon \mathcal{C}\rightarrow \mathcal{D}$ contained in $\mathcal{C}of\cap\mathcal{Q}equ$ is a retract of some dg functor in $\mathcal{J}\mbox{-cell}$.

Consider a factorization $\mathcal{C}\stackrel{j}\rightarrow\mathcal{C}' \stackrel{q}\rightarrow \mathcal{D}$ of $F$ with $j\in \mathcal{J}\mbox{-cell} $ and $q\in \mathcal{J}^\perp=\mathcal{F}isof$. Since $j$ is a quasi-equivalence, so is $q$. Since $F\in \mathcal{C}of$ and $q\in \mathcal{Q}equ\cap \mathcal{F}isof$, we have a lifting $H$ in the following diagram.
\[
\xymatrix{\mathcal{C}\ar[d]_-F \ar[rr]^-{j} && \mathcal{C}'\ar[d]^-q\\
         \mathcal{D} \ar@{.>}[rru]^-H \ar@{=}[rr] && \mathcal{D}}
\]
Then $F$ is a retract of $j\in \mathcal{J}\mbox{-cell} $, as is shown in the following commutative diagram.
\[
\xymatrix{
\mathcal{C}\ar[d]_-F \ar@{=}[r] & \mathcal{C}\ar[d]^-j \ar@{=}[r] &\mathcal{C}\ar[d]^-F\\
\mathcal{D} \ar[r]^-H & \mathcal{C}' \ar[r]^-q &\mathcal{D}
}
\]
This completes the whole proof. \hfill $\square$

\subsection{A path object} We denote by ${\rm Hodgcat}$ the homotopy category of ${\rm DGCat}$ with respect to the Dwyer-Kan model structure. We observe that every dg category is fibrant, and a dg category is cofibrant if and only if it is a retract of a semi-free dg category.

Let $\mathcal{B}$ be a dg category. Following \cite[Subsection~2.9]{Dri}, the \emph{morphism dg category} $mor(\mathcal{B})$ is defined as follows. Its objects are triples $(X_1, X_0; f)$, where $X_i$ are objects in $\mathcal{B}$ and $f\colon X_1\rightarrow X_0$ is a closed morphism of degree zero. The Hom complex between $(X_1, X_0; f)$ and $(Y_1, Y_0; g)$ is given by
$$mor(\mathcal{B})((X_1, X_0; f), (Y_1, Y_0; g))=\begin{pmatrix}
\mathcal{B}(X_0, Y_0) & \Sigma^{-1}\mathcal{B}(X_1, Y_0)\\
0 & \mathcal{B}(X_1, Y_1)
\end{pmatrix},$$
whose typical element of degree $p$ is of the form
$$\begin{pmatrix}
a_0 & s^{-1}h\\
0 & a_1
\end{pmatrix}$$
with $|a_0|=p=|a_1|$ and $|h|=p-1$. We visualize the morphism as the following diagram.
\[\xymatrix{
X_1\ar[d]_-{a_1} \ar@{.>}[drr]^-{h}\ar[rr]^-f && X_0\ar[d]^-{a_0}\\
Y_1\ar[rr]^-g && Y_0
}\]

 The differential of morphisms is given by
$$d_{mor(\mathcal{B})}\begin{pmatrix}
a_0 & s^{-1}h\\
0 & a_1
\end{pmatrix}=\begin{pmatrix}
d_\mathcal{B}(a_0) & s^{-1}(-d_\mathcal{B}(h)+a_0\circ f-g\circ a_1)\\
0 & d_\mathcal{B}(a_1)
\end{pmatrix}.$$
For another morphism $\begin{pmatrix}b_0& s^{-1}e\\
                                      0 & b_1\end{pmatrix}\colon (Y_1, Y_0;g)\rightarrow (Z_1, Z_0;k)$, the composition is defined by
$$\begin{pmatrix}b_0& s^{-1}e\\
                                      0 & b_1\end{pmatrix}\ \circ \begin{pmatrix}
a_0 & s^{-1}h\\
0 & a_1
\end{pmatrix}=\begin{pmatrix}
b_0\circ a_0 & s^{-1}((-1)^{|b_0|}b_0\circ h+e\circ a_1)\\
0 & b_1\circ a_1
\end{pmatrix}.$$

Recall the homotopy category $H^0(\mathcal{B})$. Denote by ${\rm mor}(H^0(\mathcal{B}))$ its morphism category. For a closed morphism $f$ of degree zero, we denote by $[f]$ the corresponding morphism in $H^0(\mathcal{B})$.

The following observation is due to \cite[Lemma~2.2.4]{CC}.

\begin{lem}\label{lem:mor-cat}
The canonical functor
$$H^0(mor(\mathcal{B}))\longrightarrow {\rm mor}(H^0(\mathcal{B})), \; (X_1, X_0; f)\mapsto (X_1, X_0; [f])$$
is full and dense, whose kernel ideal is square zero.
\end{lem}

\begin{proof}
The only nontrivial part is the last statement. For this, we take two composable morphisms in the kernel ideal: $\begin{pmatrix}
d_\mathcal{B}(u_0) & s^{-1}h\\
0 & d_\mathcal{B}(u_1)
\end{pmatrix}\colon (X_1, X_0; f) \rightarrow (Y_1, Y_0;g)$ and $\begin{pmatrix}
d_\mathcal{B}(v_0) & s^{-1}e\\
0 & d_\mathcal{B}(v_1)
\end{pmatrix}\colon (Y_1, Y_0; g) \rightarrow (Z_1, Z_0;k)$. Here, the morphisms $u_i, v_i,h$ and $e$ are of degree $-1$. Using the conditions $d_\mathcal{B}(h)=d_\mathcal{B}(u_0)\circ f-g\circ d_\mathcal{B}(u_1)$ and $d_\mathcal{B}(e)=d_\mathcal{B}(v_0)\circ g-k\circ d_\mathcal{B}(v_1)$, we have
$$\begin{pmatrix}
d_\mathcal{B}(v_0) & s^{-1}e\\
0 & d_\mathcal{B}(v_1)
\end{pmatrix}\circ \begin{pmatrix}
d_\mathcal{B}(u_0) & s^{-1}h\\
0 & d_\mathcal{B}(u_1)
\end{pmatrix}= d_{mor(\mathcal{B})}  \begin{pmatrix}
v_0\circ d_\mathcal{B}(u_0) & s^{-1}h'\\
0 & d_\mathcal{B}(v_1)\circ u_1
\end{pmatrix},$$
where $h'=e\circ u_1-v_0\circ h-v_0\circ g\circ u_1$. Therefore, the corresponding composition in $H^0(mor(\mathcal{B}))$ is zero.
\end{proof}

For $i=0, 1$, we have the projection dg functors $\pi_i\colon mor(\mathcal{B})\rightarrow \mathcal{B}$ sending $(X_1, X_0; f)$ to $X_i$.

\begin{cor}\label{cor:mor-cat}
Let $\begin{pmatrix}
a_0 & s^{-1}h\\
0 & a_1
\end{pmatrix}\colon (X_1, X_0; f)\rightarrow (Y_1, Y_0; g)$ be a closed morphism of degree zero. Then it is a homotopy equivalence in $mor(\mathcal{B})$ if and only if both $a_i$ are homotopy equivalences in $\mathcal{B}$.
\end{cor}

\begin{proof}
By applying the dg functors $\pi_i$, we infer the ``only if" part. For the ``if" part, we observe that $$([a_1], [a_0])\colon (X_1, X_0; [f])\longrightarrow (Y_1, Y_0; [g])$$
is an isomorphism in ${\rm mor}(H^0(\mathcal{B}))$. By Lemma~\ref{lem:mor-cat}, the canonical functor reflects isomorphisms. Then we infer that the given morphism represents an isomorphism in $H^0(mor(\mathcal{B}))$, that is, it is a homotopy equivalence in $mor(\mathcal{B})$.
\end{proof}

Denote by $\mathcal{P}(\mathcal{B})$ the full dg subcategory of $mor(\mathcal{B})$ formed by all objects $(X_1, X_0; f)$ with $f\colon X_1\rightarrow X_0$ a homotopy equivalence. The following result is due to \cite[Proposition~2.0.11]{Tab10}; compare \cite[Subsection~D.1]{Dri}.

\begin{prop}\label{prop:dg-pathobj}
The following diagram
$$\mathcal{B}\xrightarrow{\rm diag} \mathcal{P}(\mathcal{B}) \xrightarrow{\begin{pmatrix} \pi_0\\ \pi_1\end{pmatrix}} \mathcal{B}\times \mathcal{B}$$
is a good path object for $\mathcal{B}$, where ${\rm diag}$ sends $X$ to $(X, X; {\rm Id}_X)$.
\end{prop}

\begin{proof}
Write $\mathcal{P}=\mathcal{P}(\mathcal{B})$. We observe that ${\rm diag}$ is quasi-fully faithful, because $\mathcal{P}((X, X; {\rm Id}_X), (Y, Y; {\rm Id}_Y))$ is the cocylinder of $\mathcal{B}(X, Y)$. For  any object $(X_1, X_0; f)$ in $\mathcal{P}$, we have that
 $$\begin{pmatrix} f & 0\\ 0 & {\rm Id}_{X_1}\end{pmatrix}\colon (X_1, X_1; {\rm Id}_{X_1})\longrightarrow (X_1, X_0; f)$$
is a homotopy equivalence; see Corollary~\ref{cor:mor-cat}. It follows that ${\rm diag}$ is a quasi-equivalence.

The functor $\begin{pmatrix} \pi_0\\ \pi_1\end{pmatrix}$ is clearly full. To prove that it is a full isofibration, we take an object $(X_1, X_0;f )$ in $\mathcal{P}$  and two homotopy equivalences $b\colon X_1\rightarrow X_1'$ and $a\colon X_0\rightarrow X_0'$. Take a homotopy inverse $e\colon X_1'\rightarrow X_1$ of $b$ and $h_{X_1}\in \mathcal{B}(X_1, X_1)^{-1}$ satisfying $e\circ b-{\rm Id}_{X_1}=d_\mathcal{B}(h_{X_1})$. Then $(X'_1, X'_0; a\circ f\circ e)$ is an object in $\mathcal{P}$. By Corollary~\ref{cor:mor-cat}, we have a homotopy equivalence
$$\begin{pmatrix} a & -a\circ f\circ h_{X_1}\\ 0 & b\end{pmatrix}\colon (X_1, X_0; f)\longrightarrow (X'_1, X'_0; a\circ f\circ e).$$
This is the required morphism, because it is sent by $\begin{pmatrix} \pi_0\\ \pi_1\end{pmatrix}$ to $(b, a)$.
\end{proof}

The above result implies the following standard fact: up to  quasi-isomorphism, the dg endomorphism algebra is a  invariant in the homotopy category.

\begin{cor}
Let $X$ and $Y$ be two objects in $\mathcal{B}$ such that they are isomorphic in $H^0(\mathcal{B})$. Then the dg endomorphism algebras $\mathcal{B}(X,X)$ and $\mathcal{B}(Y, Y)$ are linked by a zigzag of quasi-isomorphisms of dg algebras.
\end{cor}

\begin{proof}
Assume that $\theta\colon X\rightarrow Y$ is a homotopy equivalence in $\mathcal{B}$. Then $(X, Y; \theta)$ is an object in $\mathcal{P}(\mathcal{B})$ satisfying $\pi_0(X, Y; \theta)=Y$ and $\pi_1(X, Y; \theta)=X$. Proposition~\ref{prop:dg-pathobj} implies that both $\pi_0$ and $\pi_1$ are quasi-equivalences. Therefore, we have a required chain
$$\mathcal{B}(X, X)\longleftarrow  \mathcal{P}(\mathcal{B})((X, Y; \theta), (X, Y; \theta))= \begin{pmatrix}
\mathcal{B}(Y, Y) & \Sigma^{-1}\mathcal{B}(X, Y)\\
0 & \mathcal{B}(X, X)
\end{pmatrix}  \longrightarrow \mathcal{B}(Y, Y)$$
consisting of quasi-isomorphisms between  dg endomorphism algebras.
\end{proof}

\begin{defn}
Two dg functors $F, G\colon \mathcal{A}\rightarrow \mathcal{B}$ are said to be \emph{cochain homotopic}, denoted by $F\stackrel{c}{\sim} G$, if there exists a dg functor $K\colon \mathcal{A}\rightarrow \mathcal{P}(\mathcal{B})$ such that $F=\pi_0 K$ and $G=\pi_1 K$.
\end{defn}

We refer to \cite[Subsection~3.3]{Kel99} and \cite[Remark~2.0.12]{Tab10} for the following observation.

\begin{lem}
Assume that $F, G\colon \mathcal{A}\rightarrow \mathcal{B}$ are two dg functors. Then $F\stackrel{c}{\sim} G$ if and only if the following conditions hold: for each object $A$ in $\mathcal{A}$, there exists a homotopy equivalence $\eta_A\colon G(A)\rightarrow F(A)$ which is in general not functorial in $A$; for any $A$ and $A'$, there exists a map of degree $-1$
$$h_{A, A'}\colon \mathcal{A}(A, A')\longrightarrow \mathcal{B}(G(A), F(A'))$$
such that
$$F(f)\circ \eta_A-\eta_{A'} \circ G(f)=h_{A, A'}(d_\mathcal{A}(f))+d_\mathcal{B}(h_{A, A'}(f)),$$
and
$$h_{A, A''}(g\circ f)=h_{A', A''}(g)\circ G(f)+(-1)^{|g|} F(g)\circ h_{A, A'}(f) $$
hold for any $f\colon A\rightarrow A'$ and $g\colon A'\rightarrow A''$ in $\mathcal{A}$.
\end{lem}

\begin{proof}
We only mention that for the ``if" part, we construct $K\colon \mathcal{A}\rightarrow \mathcal{P}(\mathcal{B})$ by $K(A)=(G(A), F(A); \eta_A)$ and $K(f)=\begin{pmatrix} F(f) & s^{-1}h_{A, A'}(f)\\
0 & G(f)\end{pmatrix}$.
\end{proof}

We denote by $\stackrel{r}{\sim}$ the right homotopy relation with respect to the Dwyer-Kan model structure.

\begin{prop}
Let $F, G\colon \mathcal{A}\rightarrow \mathcal{B}$ be two dg functors. Then the following statements hold.
\begin{enumerate}
\item $F\stackrel{c}{\sim} G \Rightarrow F\stackrel{r}{\sim} G$.
\item Assume that $\mathcal{A}$ is cofibrant. Then $\stackrel{c}{\sim}$ and $\stackrel{r}{\sim}$ coincide.
\end{enumerate}
\end{prop}

\begin{proof}
(1) is trivial, and (2) follows from Remark~\ref{rem:homo-fixed}.
\end{proof}

\vskip 5pt

\noindent {\bf Acknowledgements.} \; The author thanks Xiaofa Chen, Haoyu Wang and  Jianuo Zhou for many helpful comments.

\end{document}